\documentclass[a4paper,10pt]{amsart}
\usepackage[english]{babel}
\usepackage[utf8x]{inputenc}
\usepackage[T1]{fontenc}
\usepackage{amsthm}
\usepackage{dsfont}
\usepackage{amssymb}
\usepackage{amsmath}
\theoremstyle{plain}
\usepackage{chngcntr}
\usepackage{relsize}
\usepackage{pdfpages}

\newtheorem{lemma}{Lemma}
\newtheorem{theorem}[lemma]{Theorem}
\newtheorem{proposition}[lemma]{Proposition}
\newtheorem{corollary}[lemma]{Corollary}

\newtheorem{remark}[lemma]{Remark}
\newtheorem{hypothesis}[lemma]{Hypothesis}
\usepackage[a4paper,top=3cm,bottom=2cm,left=3cm,right=3cm,marginparwidth=1.75cm]{geometry}

\usepackage{float}
\usepackage{nccmath}
\usepackage{mathtools}
\usepackage{amsfonts}
\usepackage{amsmath}
\usepackage{graphicx}
\usepackage[colorinlistoftodos]{todonotes}
\usepackage[colorlinks=true, allcolors=purple]{hyperref}
\makeatletter
\newcommand*{\rom}[1]{\expandafter\@slowromancap\romannumeral #1@}
\makeatother

\title{Partial progress towards the large sieve for square moduli}

\author{Stephan Baier}
\address{Stephan Baier,
Ramakrishna Mission Vivekananda Educational and Research Institute, Department of Mathematics, G. T. Road, PO Belur Math, Howrah, West Bengal 711202, India}
\email{stephanbaier2017@gmail.com}

\begin{document}
\maketitle
\begin{abstract} This paper condenses and expands upon our recent contributions arXiv:2603.00768 and arXiv:2603.25814, which belong together, producing a single streamlined manuscript. We have added consequences of an analogue of Hooley's hypothesis R$^{\ast}$ for short Sali\'e sums. The motivation for these contributions comes from the large sieve for square moduli, which presents a specialized but compelling research object. The first part arXiv:2601.15448 of this series used geometry of numbers and established new conditional results on the large sieve for square moduli under reasonable hypotheses on additive energies of modular square roots. In contrast, we here explore how much can be achieved unconditionally by using purely Fourier analytic techniques. While our progress is limited, it is worthwhile to exhaust these classical techniques first before moving to other ideas. This may be of relevance for a potential future resolution of this problem (particularly in the function field setting). The central problem is to estimate the number of fractions with square moduli in short intervals.
\end{abstract}

\tableofcontents

\section{Introduction and main results}
\subsection{General notation}
\begin{itemize}
\item Throughout this paper, following usual custom, we assume that $\varepsilon$ is an arbitrarily small positive number which may change from line to line. All implied constants are allowed to depend on $\varepsilon$. 
\item For a real number $x$ and natural number $r$, we set 
$$
e(x):=e^{2\pi i x} \quad \mbox{and} \quad e_r(x):=e\left(\frac{x}{r}\right). 
$$
\item The symbols $\mathbb{N}$, $\mathbb{Z}$, $\mathbb{R}$ and $\mathbb{C}$ stand for the sets of natural numbers, integers, real numbers and complex numbers, respectively. 
\item We denote by $\mathbb{R}_{\ge 0}$ the set of non-negative real numbers and by $\mathbb{R}_{>0}$ the set of positive real numbers. 
\item Let $\mathcal{T}\subset \mathbb{N}$ be a finite set. We abbreviate finite sequences $(\alpha_t)_{t\in \mathcal{T}},(\beta_t)_{t\in \mathcal{T}}$ by bold letters $\boldsymbol{\alpha}$,$\boldsymbol{\beta}$,.... 
and write
$$
||\boldsymbol{\alpha}||_{\infty}:=\sup\limits_{t\in \mathcal{T}} |\alpha_t| \quad \mbox{and} \quad
||\boldsymbol{\alpha}||_p:=\left(\sum\limits_{t\in \mathcal{T}} |\alpha_t|^p\right)^{1/p} \mbox{ for } p\ge 1.
$$
\item We denote the greatest common divisor of two integers $a,b$, not both equal to zero, by $(a,b)$.
\item  We denote the least common multiple of two integers $a,b$, not both equal to zero, by $[a,b]$. 
\item The functions $\lfloor x \rfloor$ and $\lceil x \rceil$ are the floor and ceiling function, respectively. 
\item The fractional part of $x\in \mathbb{R}$ is defined as $\{x\}:=x-\lfloor x\rfloor$. 
\item The distance of $x\in \mathbb{R}$ to the nearest integer is denoted as $||x||$.
\item If $\mathcal{S}$ is a set, then we denote by $\mathbf{1}_{\mathcal{S}}$ its indicator function, defined as 
$$
\mathbf{1}_{\mathcal{S}}(x):=\begin{cases} 1 & \mbox{ if } x\in \mathcal{S},\\ 0& \mbox{ otherwise.}\end{cases}
$$
\item The M\"obius, Euler totient and divisor function are denoted by $\mu(n)$, $\varphi(n)$ and $\tau(n)$, respectively.
\item If $c\in \mathbb{N}$ is a modulus and $a\in \mathbb{Z}$ is coprime to $c$, then we denote by $\overline{a}$ a multiplicative inverse of $a$, i.e., $\overline{a}$ satisfies $\overline{a}a\equiv 1 \bmod{c}$. In exponentials, we also write 
sometimes  $e_c\left(\frac{A}{a}\right)$
in place of $e_c(A\overline{a})$. 
\item For odd $c\in \mathbb{N}$ and $n\in \mathbb{Z}$, the Jacobi symbol is denoted by $\left(\frac{n}{c}\right)$.
\end{itemize}

\subsection{Large sieve with square moduli} 
The classical large sieve with additive characters, a tool of fundamental importance in analytic number theory, is the following inequality (see \cite[Satz 5.2.2]{Bru}).
\begin{theorem}
For any $Q,N\ge 1$, $M\in \mathbb{R}$ and any finite sequence $(a_n)_{M<n\le M+N}$ of complex numbers, we have
\begin{equation} \label{classicallargesieve}
\sum\limits_{q\le Q}\sum\limits_{\substack{a=1\\ (q,a)=1}}^{q} \Big|\sum\limits_{M<n\le M+N} a_ne_q\left(na\right)\Big|^2\le (N+Q^2-1)Z,
\end{equation}
where 
\begin{equation}  \label{Znote}
Z:=\sum\limits_{M<n\le M+N} |a_n|^2.
\end{equation}
\end{theorem}
We will keep the notation of $Z$ in \eqref{Znote} throughout this paper. 

It is known that the inequality \eqref{classicallargesieve} is sharp if $Q$ and $N$ are positive integers (see \cite[page 159]{Bru}). There are, however, a lot of open problems on modified versions of this inequality. One interesting question is to what extent this bound can be improved if the set of moduli is restricted to a sparse subset of the natural numbers. Zhao initiated the study of the large sieve with square moduli in his article \cite{Zhao1}, the problem being to obtain as strong as possible bounds for the quantity
\begin{equation} \label{lsdefin}
S\left(Q,M,N,(a_n)\right):=\sum\limits_{q\le Q}\sum\limits_{\substack{a=1\\ (q,a)=1}}^{q^2} \Big|\sum\limits_{M<n\le M+N} a_ne\left(n\cdot \frac{a}{q^2}\right)\Big|^2.
\end{equation}
The best known result in this direction is the following inequality due to Zhao and the author of the present article (see \cite[Theorem 1]{BaiZhao1}).
\begin{theorem} \label{lsbz}
For any $\varepsilon>0$, $Q,N\ge 1$, $M\in \mathbb{R}$ and any finite sequence $(a_n)_{M<n\le M+N}$ of complex numbers, we have
\begin{equation} \label{thebound}
S\left(Q,M,N,(a_n)\right)
\ll_{\varepsilon} (QN)^{\varepsilon}\left(Q^3+N+\min\left\{Q^2N^{1/2},Q^{1/2}N\right\}\right)Z.
\end{equation} 
\end{theorem} 
This result has found a lot of arithmetic applications (see \cite{BaiZhao},  \cite{BPS}, \cite{BFKS},  \cite{Mat}, \cite{Mer}, \cite{SZ},  \cite{TuPa}). It is stronger than Zhao's initial bound  
\begin{equation} \label{original}
S\left(Q,M,N,(a_n)\right)\ll (QN)^{\varepsilon}\left(Q^3+Q^2N^{1/2}+Q^{1/2}N\right)Z
\end{equation}
in \cite{Zhao1} unless $N^{1-\varepsilon}\le Q^3\le N^{1+\varepsilon}$. In this situation when $Q^3$ is close to $N$, there has been no unconditional improvement of \eqref{original} for the last 20 years. Zhao \cite{Zhao1} conjectured a bound of 
$$
S\left(Q,M,N,(a_n)\right)\ll (QN)^{\varepsilon}(Q^3+N)Z.
$$  
At the "critical" point $N=Q^3$, this conjecture predicts a bound of $\ll Q^{\varepsilon}NZ$ whereas the established bound \eqref{thebound} gives 
\begin{equation} \label{Q3bound}
S\left(Q,M,Q^3,(a_n)\right)\ll Q^{1/2+\varepsilon}NZ.
\end{equation}  
In our recent article \cite{Baier}, we derived a conditional improvement of \eqref{Q3bound} under a reasonable hypothesis on higher additive energies of modular square roots. The motivation of the present paper is the problem of  improving this bound unconditionally to 
\begin{equation} \label{desired}
S\left(Q,M,Q^3,(a_n)\right)\ll Q^{1/2-\eta}NZ \quad \mbox{for some } \eta>0.
\end{equation}
We will make limited progress towards this goal using Fourier analysis. It is worthwhile to exhaust these classical techniques first before moving to other ideas such as those indicated in \cite[section 5]{Baier}. Our problem is directly linked to bounding the number of Farey fractions with square denominators in short intervals, as described in the following. 

Set $\Delta:=1/N$ and 
\begin{equation*} 
P(\alpha):=\sharp\left\{(q,a)\in \mathbb{Z}^2: 1\le q\le Q, \ (q,a)=1,\ \left|a/q^2-\alpha\right|\le \Delta\right\}\quad \mbox{for } \alpha\in \mathbb{R},
\end{equation*}
which counts Farey fractions $a/q^2$ appearing on the right-hand side of \eqref{lsdefin} in the $\Delta$-ball centred at $\alpha$. It is well-known that $S\left(Q,M,N,(a_n)\right)$ and $P(\alpha)$ are related by the inequality
$$
S\left(Q,M,N,(a_n)\right)\ll NZ\max\limits_{0\le \alpha\le 1} P(\alpha)
$$
(see \cite[Lemma 1]{BaiFirst}, for example). Approximating $\alpha$ using the Dirichlet approximation theorem by fractions with small denominators,  it is easy to conclude the following (see \cite[Lemma 2]{BaiFirst}).
\begin{lemma}
We have 
 \begin{equation} \label{Prelationmod}
S\left(Q,M,N,(a_n)\right)
\ll NZ\max\limits_{\substack{r\in \mathbb{N}\\ 1\le r\le \sqrt{N}}}\ \max\limits_{\substack{b\in \mathbb{Z}\\ (b,r)=1}}\ \max\limits_{\Delta\le z\le \sqrt{\Delta}/r} P\left(b/r+z\right). 
\end{equation}
\end{lemma} 
This leaves us with estimating $P(b/r+z)$. As demonstrated in \cite{BaiZhao1}, this quantity can be converted into expressions involving character sums to the modulus $r$. In this paper, we assume that 
$$
\boxed{
N=Q^3.}
$$ 
To achieve our desired bound \eqref{desired}
in this situation, we have to show that 
\begin{equation} \label{Paim}
P\left(b/r+z\right)\ll Q^{1/2-\eta} \quad \mbox{for some } \eta>0
\end{equation}
whenever $1\le r\le Q^{3/2}$, $(b,r)=1$, $Q^{-3}\le z\le Q^{-3/2}r^{-1}$.
In \cite[sections 9 and 10]{BaiLS}, the bound \eqref{Paim} was unconditionally achieved using exponential sum techniques in the range 
\begin{equation} \label{uncondrange}
1\le r\le Q^{15/26-\varepsilon},
\end{equation}
with $\varepsilon>0$ depending on $\eta$.
The range $Q^{1/2+\varepsilon}\le r\le Q^{3/2}$, which overlaps with the range in \eqref{uncondrange} if $\varepsilon$ is small enough, was handled conditionally under a hypothesis on higher additive energies of modular square roots in \cite{Baier}. The theorem below is our main result in the present paper. It extends the set of $r$'s for which we establish \eqref{Paim} unconditionally and provides another conditional result in addition. 

\begin{theorem}\label{extendedrange} 
Suppose that
\begin{equation} \label{supposedrange} 
N=Q^3, \quad Q^{1/2+\varepsilon}\le r\le Q^{3/2}, \quad (b,r)=1, \quad Q^{-3}\le z\le Q^{-3/2}r^{-1}.
\end{equation}
Then the following hold.\medskip\\
{\rm (i)} If $r$ is odd and squarefree, then 
\begin{equation} \label{roddsquarefree}
P\left(b/r+z\right)\ll \left(Q^{5/8}r^{-1/4}+Q^{1/4}r^{1/4}\right)Q^{\varepsilon}.
\end{equation}
{\rm (ii)} If $r=p^2$ is a prime square, then 
\begin{equation} \label{primesquare} 
P\left(b/r+z\right)\ll \left(Q^{5/8}r^{-1/4}+Q^{1/4}r^{1/4}+Q^{3/8}r^{1/8}\right)Q^{\varepsilon}.
\end{equation}
{\rm (iii)} Under Hooley's Hypothesis R$^{\ast}$ for short Sali\'e sums, we have 
\begin{equation} \label{roddsquarefreecond}
P\left(b/r+z\right)\ll Q^{9/16+\varepsilon}r^{-1/8}
\end{equation}
if $r=p$ is a prime or $r=p^2$ is a prime square.
\end{theorem}
By the results in parts (i) and (ii) of the above Theorem \ref{extendedrange}, we achieve a bound of the form in \eqref{Paim} unconditionally if $r$ is an odd squarefree number or a prime square and lies in the range $Q^{1/2+\varepsilon}\le r\le Q^{1-\varepsilon}$. By the result in part (iii), this range can be extended to the desired range $Q^{1/2+\varepsilon}\le r\le Q^{3/2}$ under Hooley's hypothesis R$^{\ast}$ for short Sali\'e sums if $r$ is a prime or a prime square (for details about this hypothesis, see section \ref{preli}). In fact, we established a generalization of part (i) for odd moduli $r$ in \cite{Baier2}. However, this result is weak if $r$ contains a large powerfull part. In particular, it becomes trivial if $r$ is a prime square. For the sake of easier reading, we will work out only the case of squarefree $r$ in part (i). Using a new technical idea in addition, we then handle prime squares in part (ii) for the same $r$-range $Q^{1/2+\varepsilon}\le r\le Q^{1-\varepsilon}$. We expect that with additional efforts, general moduli $r$ in this range can be managed along similar lines, but the technical details will become complicated.  Similarly, at the cost of more involved calculations, it may be possible to prove \eqref{Paim} for all $r\le Q^{3/2}$ under Hooley's hypothesis R$^{\ast}$ for short Sali\'e sums, thus achieving another conditional improvement of Theorem \ref{lsbz}. We point out that a function field version of Theorem \ref{lsbz} was worked out in \cite{BaiSi}, and a generalized version of Hooley's hypothesis R$^{\ast}$ for function fields was established by Sawin and Shusterman in their recent breakthrough paper \cite{SaSh}. So potentially, it might become possible to prove \eqref{desired} unconditionally for function fields. However, Sawin and Shusterman currently require the modulus to be squarefree, which creates a limitation.   

\subsection{Link to bilinear sums with modular square roots}    
We refer to a solution $k$ of the congruence $k^2\equiv t\bmod{r}$ as a modular square root of $t$ modulo $r$. By abuse of notation, we denote by $\sqrt{t}$ the {\it collection} of all modular square roots of $t$ modulo $r$, if existent. Consequently, given $r\in \mathbb{N}$, a set $\mathcal{T}$ of integers, a periodic function $g:\mathbb{Z}\rightarrow \mathbb{C}$ with period $r$ and a function $h:\mathbb{Z}\rightarrow \mathbb{C}$, we understand $\sum_{t\in \mathcal{T}} g(\sqrt{t})h(t)$ as the sum of {\it all} expressions $g(\sqrt{t})h(t)$ with {\it any} modular square roots $\sqrt{t}$ modulo $r$ of elements $t$ in $\mathcal{T}$, that is,  
$$
\sum\limits_{t\in \mathcal{T}} g(\sqrt{t})h(t)=\sum\limits_{t\in \mathcal{T}} \sum\limits_{\substack{k \bmod{r}\\ k^2\equiv t\bmod{r}}} g(k)h(t). 
$$  

The following lemma provides a link between the quantity $P(b/r+z)$ and bilinear sums with modular square roots which was established in \cite[section 6]{BaiLS}. 

\begin{lemma} \label{Pxlemma} Suppose the conditions in \eqref{supposedrange} are satisfied. Let $j$ be a multiplicative inverse of $b$ modulo $r$, i.e., $bj\equiv 1\bmod{r}$. Write 
\begin{equation} \label{ranges}
z=\frac{1}{Q^{3/2+\gamma}r} \quad \mbox{with } \gamma\ge 0.
\end{equation}
Then for suitable Schwartz class functions $W:\mathbb{R}\rightarrow \mathbb{C}$ and $V: \mathbb{R}\rightarrow \mathbb{R}_{\ge 0}$, where $V$ has compact support in $\mathbb{R}_{>0}$, the inequality 
\begin{equation} \label{specialPx}
\begin{split}
P\left(\frac{b}{r}+z\right)\ll & 1+\frac{\delta}{Qr}\cdot \Big| \sum\limits_{l\in \mathbb{Z}} W\left(\frac{l\delta}{Qr}\right)\sum\limits_{m\in \mathbb{Z}}V\left(\frac{m}{Q^{1/2-\gamma}}\right)e\left(-\frac{l\sqrt{m}}{r}\cdot Q^{3/4+\gamma/2}\right) e_r\left(l\sqrt{jm}\right)\Big|
\end{split}
\end{equation}
holds for any $\delta$ in the range
\begin{equation} \label{specialdeltarange}
Q^{1/2+\gamma}r\le \delta\le Q^2. 
\end{equation}
Above, $\sqrt{m}$ denotes the ordinary square root of the real number $m$ in the term
$e\left(-l\sqrt{m}Q^{3/4+\gamma/2}/r\right)$,
and $\sqrt{jm}$ denotes a modular square root of $jm$ modulo $r$, if existent, in the term $e_r\left(l\sqrt{jm}\right)$. (As indicated at the beginning of this subsection, we sum up the contributions of {\it all} possible modular square roots.) 
\end{lemma}

\begin{proof}
This follows from \cite[inequality (38)]{BaiLS} under the condition on $\delta$ in \cite[inequality (24)]{BaiLS} in our special situation when $N=Q^3$. 
\end{proof}

We note that for the range in \eqref{specialdeltarange} to be non-empty, we need
\begin{equation} \label{rcond1}
r\le Q^{3/2-\gamma}.
\end{equation}
Moreover, if $\delta\le Qr$, then estimating the right-hand side of \eqref{specialPx} trivially gives $P(x)\ll Q^{1/2-\gamma}$. Hence, to establish \eqref{Paim}, we need to beat the trivial bound only if $0\le \gamma<\eta$.  

\subsection{Bounds for bilinear sums with modular square roots}
To estimate the right-hand side of \eqref{specialPx}, we consider, more generally, bilinar exponential sums of the form
\begin{equation} \label{Sigmadeff}
\Sigma_f:=\Sigma_f(r,j,L,M,\boldsymbol{\alpha},\boldsymbol{\beta}):=\sum\limits_{|l|\le L}\sum\limits_{1\le m\le M} \alpha_l\beta_m e_r\left(l\sqrt{jm}\right)e(lf(m)),
\end{equation}
where $f: [1,M]\rightarrow \mathbb{R}$ is a continuously differentiable function. If $f\equiv 0$, we set
\begin{equation*}
\Sigma_0:=\Sigma_0(r,j,L,M,\boldsymbol{\alpha},\boldsymbol{\beta}):=\sum\limits_{|l|\le L}\sum\limits_{1\le m\le M} \alpha_l\beta_m e_r\left(l\sqrt{jm}\right).
\end{equation*}
To prove Theorem \ref{extendedrange}, we establish the following. 
\begin{theorem}\label{bilinearboundnew}  Suppose that $r,j\in \mathbb{N}$, $(r,j)=1$ and $r^{\varepsilon}\le L,M\le r^{1-\varepsilon}$. Let $f:[1,M]\rightarrow \mathbb{R}$ be a continuously differentiable function such that $|f'(x)|\le F$ on $[1,M]$, where 
$F\le L^{-1}$. Let $\boldsymbol{\alpha}=(\alpha_l)_{|l|\le L}$ and $\boldsymbol{\beta}=(\beta_m)_{1\le m\le M}$ be any finite sequences of complex numbers and suppose that $H\in \mathbb{N}$ satisfies 
\begin{equation} \label{Hrange}
r^{\varepsilon}\le H\le \min\left\{\frac{1}{LF},M\right\}.
\end{equation}
{\rm (i)} If $r$ is odd and squarefree, then
\begin{equation} \label{generalbilinearboundnew}
\Sigma_f \ll \left(H^{-1/2}L^{1/2}M+L^{1/2}M^{1/2}r^{1/4}+M\right)||\boldsymbol{\alpha}||_2 ||\boldsymbol{\beta}||_{\infty}r^{\varepsilon}. 
\end{equation}
{\rm (ii)} If $r:=p^2$ is a prime square, then 
\begin{equation} \label{mainbound1}
\Sigma_f \ll \left(H^{-1/2}L^{1/2}M+L^{1/2}M^{1/2}r^{1/4}+M+H^{-1/2}L^{1/2}M^{1/2}r^{3/8}\right)||\boldsymbol{\alpha}||_2 ||\boldsymbol{\beta}||_{\infty}r^{\varepsilon}.
\end{equation}
{\rm (iii)} Under Hooley's hypothesis R for short Sali\'e sums, we have 
\begin{equation} \label{generalbilinearboundnewR}
\begin{split}
&\Sigma_f\ll \left(H^{-1/4}L^{1/2}M+M^{1/2}r^{1/4}+M\right)||\boldsymbol{\alpha}||_2 ||\boldsymbol{\beta}||_{\infty}r^{\varepsilon}
\end{split}
\end{equation}
if $r=p$ is a prime and  
\begin{equation} \label{mainbound1R}
\begin{split}
\Sigma_f\ll \Big( & H^{-1/4}L^{1/2}M+M^{1/2}r^{1/4}+M+L^{1/2}M^{1/2}r^{-1/8}+H^{-1/2}M^{1/2}r^{3/8}+\\
& H^{-1/2}L^{1/2}M^{1/2}r^{1/4}\Big) \cdot ||\boldsymbol{\alpha}||_2 ||\boldsymbol{\beta}||_{\infty}r^{\varepsilon} 
\end{split}
\end{equation}
if $r=p^2$ is a prime square.
\end{theorem}
Taking $f\equiv 0$ and $H=M$ in Theorem \ref{bilinearboundnew} implies the following. 

\begin{corollary} \label{bilinearcornew} Suppose that $r,j\in \mathbb{N}$, $(r,j)=1$ and $r^{\varepsilon}\le L,M\le r^{1-\varepsilon}$. Let 
$\boldsymbol{\alpha}=(\alpha_l)_{|l|\le L}$ and $\boldsymbol{\beta}=(\beta_m)_{1\le m\le M}$ be any sequences of complex numbers.\medskip\\ 
{\rm (i)} If $r$ is odd and squarefree, then
\begin{equation} \label{C1} 
\Sigma_0 \ll \left(L^{1/2}M^{1/2}r^{1/4}+M\right)||\boldsymbol{\alpha}||_2||\boldsymbol{\beta}||_{\infty}r^{\varepsilon}.
\end{equation}
{\rm (ii)} If $r=p^2$ is a prime square, then
\begin{equation} \label{C2}
\Sigma_0\ll \left(L^{1/2}M^{1/2}r^{1/4}+M+L^{1/2}r^{3/8}\right)||\boldsymbol{\alpha}||_2 ||\boldsymbol{\beta}||_{\infty}r^{\varepsilon}. 
\end{equation}
{\rm (iii)} Under Hooley's hypothesis R for short Sali\'e sums, we have
\begin{equation}  \label{Hooleyp}
\Sigma_0\ll \left(L^{1/2}M^{3/4}+M^{1/2}r^{1/4}+M\right)||\boldsymbol{\alpha}||_2 ||\boldsymbol{\beta}||_{\infty}r^{\varepsilon}
\end{equation}
if $r=p$ is a prime and  
\begin{equation} \label{Hooleyp2}
\Sigma_0\ll \left(L^{1/2}M^{3/4}+M^{1/2}r^{1/4}+M+r^{3/8}+L^{1/2}r^{1/4}\right) ||\boldsymbol{\alpha}||_2 ||\boldsymbol{\beta}||_{\infty}r^{\varepsilon}
\end{equation}
if $r=p^2$ is a prime square.
\end{corollary}
The estimates in \eqref{C1} and \eqref{C2} are non-trivial if $M\ge r^{1/2+\varepsilon}$,  the estimate \eqref{Hooleyp} is non-trivial if $LM\ge r^{1/2+\varepsilon}$, and the estimate \eqref{Hooleyp2} is non-trivial if $LM\ge r^{1/2+\varepsilon}$, $LM^2\ge r^{3/4+\varepsilon}$ and $M\ge r^{1/4+\varepsilon}$.

Our method used to establish Theorem \ref{bilinearboundnew} employs classical Fourier analysis and evaluations of Gauss and other complete exponential sums.  

\subsection{Generalized bilinear sums}
We complement our results above with bounds for generalized bilinear sums of the form 
\begin{equation*} 
\Sigma_f(p,j,\mathcal{X},\mathcal{Y},\boldsymbol{\alpha}):=\sum\limits_{l\in \mathcal{X}}\sum\limits_{m\in \mathcal{Y}} \alpha_le_p\left(l\sqrt{jm}\right)e(lf(m))
\end{equation*}
and
\begin{equation*} 
\Sigma_0(p,j,\mathcal{X},\mathcal{Y},\boldsymbol{\alpha}):=\sum\limits_{l\in \mathcal{X}}\sum\limits_{m\in \mathcal{Y}} \alpha_le_p\left(l\sqrt{jm}\right),
\end{equation*}
subject to the following conditions.
\begin{equation} \label{newcondis}
\begin{cases}
& p \mbox{ is a prime},\\
& j \mbox{ is an integer coprime to } p,\\ 
& \mathcal{X} \mbox{ is a subset of } \mathbb{F}_p\setminus \{0\}  \mbox{ of cardinality } X\ge 1,\\
& \mathcal{Y}=(A,A+Y] \mbox{ is an interval of length } Y\in \mathbb{N} \mbox{ satisfying } 1\le Y<p,\\
& \boldsymbol{\alpha}=(\alpha_l)_{l\in \mathcal{X}}\mbox{ is an arbitrary sequence of complex numbers},\\
& f:[A-2Y,A+2Y]\rightarrow \mathbb{R} \mbox{ is a continuously differentiable function},\\ 
& F:=\sup\limits_{A-2Y\le y\le A+2Y} |f'(y)|. 
\end{cases}
\end{equation}
In contrast to the previously considered bilinear sums, $l$ now runs over a general set rather than an interval. We note, however, that the $m$-sum is smooth and over an arbitrary interval. We shall establish the following result, which does not give anything non-trivial for the large sieve with square moduli but is of independent interest.

\begin{theorem} \label{mainresult}  Suppose the conditions in \eqref{newcondis} are satisfied. Let $n$ be any natural number. Suppose that $\mathcal{X}$ is included in an interval $\mathcal{I}$ of length $L\in \mathbb{N}$ satisfying
\begin{equation} \label{Lcond}
1\le L\le \frac{1}{Fp^{1/n+1/4+\varepsilon}}.
\end{equation}
Suppose further that  
\begin{equation} \label{Ycond}
Y\ge p^{1/n+1/4+\varepsilon}.
\end{equation}
Then we have   
\begin{equation} \label{mainbound}
\begin{split}
\Sigma_f(p,j,\mathcal{X},\mathcal{Y},\boldsymbol{\alpha}) \ll_n  ||\boldsymbol{\alpha}||_1^{1-1/n}||\boldsymbol{\alpha}||_{\infty}^{1/n}
X^{1/2n}Y^{1-1/2n}\left(X+\left(LF+Y^{-1}\right)p^{1+1/n}\right)^{1/2n}p^{\varepsilon}.
\end{split}
\end{equation}
\end{theorem}

Taking $f:=0$ and $\mathcal{I}:=(0,p]$ in Theorem \ref{mainresult} implies the following.

\begin{corollary} \label{Cor}  Suppose the conditions in \eqref{newcondis} are satisfied. Let $n\in \mathbb{N}$ and suppose that \eqref{Ycond} is satisfied. Then we have 
\begin{equation} \label{specialbilinearbound} 
\begin{split}
\Sigma_0(p,j,\mathcal{X},\mathcal{Y},\boldsymbol{\alpha})
\ll_n ||\boldsymbol{\alpha}||_{1}^{1-1/n}||\boldsymbol{\alpha}||_{\infty}^{1/n} X^{1/2n}
Y^{1-1/2n}\left(X+Y^{-1}p^{1+1/n}\right)^{1/2n}p^{\varepsilon}.
\end{split}
\end{equation}
\end{corollary}

Our method used to establish Theorem \ref{mainresult} follows closely that of Bag and Shparlinski in \cite{BaSh} who proved an analogous result (\cite[Theorem 3.1]{BaSh}) on bilinear sums of the form 
$$
\sum\limits_{l\in \mathcal{X}}\sum\limits_{m\in \mathcal{Y}} \alpha_l e_p\left(alm^s\right)
$$
with $s\in \mathbb{Z}$.
We believe, however, that their result needs a slight correction: As a consequence of our Remark \ref{rema} in section \ref{alter}, the right-hand side of the bound claimed in \cite[Theorem 3.1]{BaSh} should be 
$$
||\boldsymbol{\alpha}||_{1}^{1-1/l}||\boldsymbol{\alpha}||_{\infty}^{1/l} X^{1/2l}
Y\left(Y^{-1/l}p^{(l+1)/2l^2}+X^{1/2l}Y^{-1/2l}\right)p^{o(1)},
$$ 
which matches our result in Corollary \ref{Cor} above. This is not a substantial problem: Their corrected result is non-trivial if $XY^2\ge p^{1+\varepsilon}$ and $Y\ge p^{\varepsilon}$ (not $X\ge p^{\varepsilon}$).
The central idea of their (and our) method is to use the "shift by $ab$-trick", which has its roots in Burgess's treatment of short character sums. Corollary \ref{Cor} gives a non-trivial result if 
$$
XY^2\ge p^{1+\varepsilon} \quad \mbox{and} \quad Y\ge p^{1/4+\varepsilon}.
$$ 
If $\mathcal{X}$ equals the set of the integers in an interval centered at the origin and $X,Y$ satisfy $X^2Y\ge p^{1+\varepsilon}$ and $Y\ge p^{1/3+\varepsilon}$, then \cite[Corollary 1]{Baier} gives a stronger result, though.   \\ \\
{\bf Acknowledgement.} The author thanks the Ramakrishna Mission Vivekananda Educational and Research Institute for excellent working conditions. 

\section{Preliminaries} \label{preli} 
Our treatment of bilinear sums will begin with applications of two key tools from the theory of exponential sums, which are Weyl differencing and the Poisson summation formula, stated below.

\begin{proposition}[Weyl differencing] \label{Weyldif}
Let $I=(a,b]$ be an interval of length $|I|=b-a\ge 1$. Let $g: I \rightarrow \mathbb{R}$ be a function and $(\beta_m)_{m\in I}$ be a sequence of complex numbers. Then for every $H\in \mathbb{N}$ with $1\le H\le |I|$, we have the bound
$$
\left|\sum\limits_{m\in I}  \beta_me(g(m))\right|^2\ll \frac{|I|}{H}\cdot \sum\limits_{\substack{m_1,m_2\in I\\ |m_1-m_2|\le H}} \gamma(m_1,m_2) e(g(m_1)-g(m_2)),
$$ 
where 
\begin{equation} \label{cdef}
\gamma(m_1,m_2):= \left(1-\frac{|m_1-m_2|}{H}\right)\beta_{m_1} \beta_{m_2}.
\end{equation}
\end{proposition}

\begin{proof}
This is a consequence of \cite[Lemma 2.5.]{GrKo}.
\end{proof} 

For a Schwartz class function  $\Phi:\mathbb{R} \rightarrow \mathbb{C}$, we define its Fourier transform $\hat{\Phi}:\mathbb{R} \rightarrow \mathbb{C}$ as 
$$
\hat{\Phi}(y):=\int\limits_{\mathbb{R}} \Phi(x)e(-xy){\rm d}x.
$$ 
For details on the Schwartz class, see \cite{StSh}. We will use the following generalized version of the Poisson summation formula.

\begin{proposition}[Poisson summation] \label{Poisson} Let $\Phi:\mathbb{R}\rightarrow \mathbb{C}$ be a Schwartz class function, $L>0$ and $\alpha\in \mathbb{R}$. Then
$$
\sum\limits_{l\in \mathbb{Z}} \Phi\left(\frac{l}{L}\right) e\left(l\alpha\right)=L\sum\limits_{n\in \mathbb{Z}} \hat{\Phi}\left(L(\alpha-n)\right).
$$
\end{proposition}

\begin{proof} This arises by a linear change of variables from the well-known basic version of the Poisson summation formula which asserts that
$$
\sum\limits_{n\in \mathbb{Z}} F(n)=\sum\limits_{n\in \mathbb{Z}} \hat{F}(n)
$$
for any Schwartz class function $F:\mathbb{R}\rightarrow \mathbb{C}$ (see \cite{StSh}).
\end{proof} 

Moreover, we note that by the rapid decay of $\hat\Phi$ in Proposition \ref{Poisson}, we have 
\begin{equation} \label{nsumRHS}
\sum\limits_{n\in \mathbb{Z}} \hat{\Phi}\left(L(\alpha-n)\right)\ll \mathbf{1}_{[0,r^{\varepsilon}L^{-1}]}\left(||\alpha||\right)+r^{-2026} \quad \mbox{if } L\ge 1.
\end{equation}

In our treatment of bilinear sums with modular square roots, we will explicitly evaluate quadratic Gauss sums, defined as 
\begin{equation} \label{usualGauss}
G(a,b,c):=\sum_{n=0}^{c-1}e_c(an^2+bn).
\end{equation}
For odd moduli, this can be achieved by applying suitable parts of the following proposition.  

\begin{proposition} \label{Gaussprop} Let $a,b\in \mathbb{Z}$ and $c,c_1,c_2\in \mathbb{N}$. If $c$ is odd, then set 
$$
\epsilon_c:=\begin{cases} 1 & \mbox{ if } c\equiv 1 \bmod{4},\\ i & \mbox{ if } c\equiv -1\bmod{4} \end{cases}
$$
The following hold.\medskip\\
{\rm (i)}  If $(a,c)|b$, then 
$$
G(a,b,c)=(a,c) \cdot G\left(\frac{a}{(a,c)},\frac{b}{(a,c)},\frac{c}{(a,c)}\right). 
$$
{\rm (ii)}  If $(a,c)\nmid b$, then $G(a,b,c)=0$.\medskip\\
{\rm (iii)} If $(2a,c)=1$,  then 
$$
G(a,b,c)=e_c\left(-\overline{4a}b^2\right)G(a,0,c). 
$$ 
{\rm (iv)} If $(2a,c)=1$, then 
$$
G(a,0,c)=\epsilon_c\cdot \left(\frac{a}{c}\right) \cdot \sqrt{c}.
$$
{\rm (v)} If $(2a,c)=1$, then 
$$
G(a,b,c)=\epsilon_c\cdot \left(\frac{a}{c}\right) \cdot e_c\left(-\overline{4a}b^2\right)\cdot \sqrt{c}.
$$
\end{proposition}

\begin{proof}
(i) Suppose that $(a,c)=d$ and $d|b$. Set $\tilde{a}:=a/d$, $\tilde{b}=b/d$ and $\tilde{c}=c/d$. Then we have  
$$
G(a,b,c)=\sum_{n=0}^{c-1}e_c(an^2+bn)=d\sum\limits_{n=0}^{\tilde{c}-1} e_{\tilde{c}}(\tilde{a}n^2+\tilde{b}n)=dG(\tilde{a},\tilde{b},\tilde{c}),
$$ 
establishing part (i). \medskip\\
(ii) Suppose that $(a,c)=d$ and $d\nmid b$. Set $\tilde{a}:=a/d$ and $\tilde{c}:=c/d$. Then we have 
\begin{equation*}
\begin{split}
G(a,b,c)= & \sum_{n=0}^{c-1}e_c(an^2+bn)\\
= & \sum\limits_{n=0}^{c-1} e_{\tilde{c}}\left(\tilde{a}n^2\right)e_c\left(bn\right)\\
= & \sum\limits_{n=0}^{\tilde{c}-1} e_{\tilde{c}}\left(\tilde{a}n^2\right)\sum\limits_{\substack{m=0\\ m\equiv n\bmod{\tilde{c}}}}^{c-1} e_c(bm)\\ 
=& 0
\end{split}
\end{equation*}
since 
$$
\sum\limits_{\substack{m=0\\ m\equiv n\bmod{\tilde{c}}}}^{c-1} e_c(bm)=\sum\limits_{k=0}^{d-1} e_c(b(k\tilde{c}+n))=e_c(bn)\sum\limits_{k=0}^{d-1} e_d(bk)=0,
$$
establishing part (ii).\medskip\\
(iii) Suppose that $(a,c)=1$ and $c$ is odd. Then quadratic completion gives
\begin{equation*}
\begin{split}
G(a,b,c)= & \sum_{n=0}^{c-1}e_c(an^2+bn)\\
= & e\left(-\frac{\overline{4a}b^2}{c}\right)\sum_{n=0}^{c-1} e_c\left(a(n+\overline{2a}b)^2\right)\\
= & e\left(-\frac{\overline{4a}b^2}{c}\right) \sum\limits_{n=0}^{c-1} e_c(an^2)\\ 
= & e\left(-\frac{\overline{4a}b^2}{c}\right) G(a,0,c),
\end{split}
\end{equation*}
establishing part (iii). \medskip\\
(iv) This can be established by decomposing $G(a,0,c)$ into a product of quadratic Gauss sums of the form $G(a',0,p^{k})$ to odd prime power moduli $p^k$ by using \cite[Lemma 7.10]{GrKo} and then applying \cite[Lemmas 7.12. and 7.13.]{GrKo}. \medskip\\
(v) This follows by combining parts (iii) and (iv). 
\end{proof}

In the case of prime square moduli, it will be of key importance to evaluate quadratic Gauss sums with a {\it coprimality restriction} as well. These are defined as 
\begin{equation} \label{Gastdefi}
G^{\ast}(a,b,c):=\sum_{\substack{n=0\\ (n,c)=1}}^{c-1}e_c(an^2+bn).
\end{equation}
We will prove the following result for prime square moduli $c=p^2$. 
 
\begin{proposition} \label{restrictedgausssums} 
Suppose that $a,b$ are integers and $p$ is an odd prime. Then we have 
\begin{equation} \label{Gast} 
G^{\ast}\left(a,b,p^2\right)=\begin{cases} e_{p^2}\left(-\overline{4a}b^2\right)p & \mbox{ if } (ab,p)=1\\
0 & \mbox{ if } (a,p)=1 \mbox{ and } (b,p)=p,\\ 
0 & \mbox{ if } (a,p)=p \mbox{ and } (b,p)=1,\\
p^{3/2}\epsilon_p\cdot \left(\frac{a_1}{p}\right)\cdot e_{p}\left(-\overline{4a_1}b_1^2\right)-p & \mbox{ if } \left(a,p^2\right)=p \mbox{ and } \left(b,p\right)=p,\\
c_{p^2}(b) & \mbox{ if } \left(a,p^2\right)=p^2,	
\end{cases}
\end{equation}
where 
$$
a_1:=\frac{a}{p} \quad \mbox{and} \quad b_1:=\frac{b}{p},
$$
and $c_{p^2}(b)$ is the Ramanujan sum, defined as
\begin{equation} \label{Ramanujan}
c_{p^2}(b):=\sum\limits_{\substack{n=1\\ (n,p)=1}}^{p^2} e_{p^2}(bn)=\begin{cases} 0 & \mbox{ if } (b,p)=1,\\ 
-p & \mbox{ if } \left(b,p^2\right)=p,\\ p^2-p & \mbox{ if } \left(b,p^2\right)=p^2. \end{cases} 
\end{equation}
\end{proposition}

\begin{proof}
The identity in the last case $(a,p^2)=p^2$ in \eqref{Gast} is immediate. To establish the identities in the remaining cases, we write
\begin{equation*}
\begin{split}
G^{\ast}\left(a,b,p^2\right) = & G\left(a,b,p^2\right)-\sum\limits_{n=0}^{p-1} e_{p^2}\left(a(np)^2+bnp\right)\\
= & G\left(a,b,p^2\right)-\sum\limits_{n=0}^{p-1} e_{p}\left(bn\right)
\end{split}
\end{equation*}
and note that
$$
\sum\limits_{n=0}^{p-1} e_{p}\left(bn\right)=\begin{cases} 0 & \mbox{ if } (b,p)=1,\\ p & \mbox{ if } (b,p)=p.\end{cases}
$$
Now the claimed identities follow from Proposition \ref{Gaussprop}. 
\end{proof}

\begin{remark} The vanishing of $G^{\ast}\left(a,b,p^2\right)$ in the case when $(a,p)=1$ and $(b,p)=p$ is the key property of this restricted Gauss sum which allows us to handle prime square moduli.
\end{remark}

Our treatments of bilinear sums with modular square roots will lead us to mixed exponential sums with rational functions to odd prime power moduli. These are bounded using results of Perel'muter \cite{Per} and Cochrane \cite{Coch}. We consider mixed exponential sums of the form
\begin{equation} \label{mathS}
\mathcal{S}(\chi, g, f, p^m):= \sum\limits_{x=1}^{p^m} \chi(g(x))e_{p^m}(f(x)),
\end{equation}
where $p^m$ is an odd prime power, $\chi$ is a multiplicative character modulo $p^m$, and $f$, $g$ are rational functions with integer coefficients. By convention, we include only those $x$ in the summation for which $f(x)$ and $g(x)$ are defined over $\mathbb{Z}/p^m\mathbb{Z}$. For the case $m=1$, we will use the following bound due to Perel'muter. 

\begin{proposition} \label{primecase}
Suppose that one of the following conditions (i) and (ii) holds. (i) If $k$ is the order of $\chi$, there is no rational function $G$ over $\mathbb{F}_p$ such that $g=G^k$ over $\mathbb{F}_p$. (ii) There is no constant $c\in \mathbb{F}_p$ and rational function $F$ over $\mathbb{F}_p$ such that $f=F^p-F+c$ over $\mathbb{F}_p$. Then
$$
|\mathcal{S}(\chi, g, f, p^m)|\le (M+N-1)\sqrt{p},
$$ 
where $N$ is the number of poles of the rational function $f$ over $\mathbb{F}_p$, counted with multiplicities, and $M$ is the number of distinct zeros and poles of the rational function $g$ over $\mathbb{F}_p$. 
\end{proposition} 

\begin{proof} This follows from \cite[Theorem 1]{Per}.\end{proof}

To state our bound for the case $m>1$, we recall some notations from \cite{Coch} below. By $a$ we denote a fixed primitive root modulo $p^2$. We then define $r\in \mathbb{Z}$ by the relation
$$
a^{p-1} =1 + rp.
$$
It follows that $a$ is a primitive root modulo $p^m$ for any exponent $m$, and $(r,p)=1$. Further, we define a $p$-adic integer $R$ by the relation 
\begin{equation} \label{Rrel}
R := p^{-1}\log_p(1 + rp) = p^{-1}\sum\limits_{i=1}^{\infty}
\frac{(-1)^{i+1}(rp)^i}{i} \equiv r \bmod{p},
\end{equation}
where $\log_p$ is the $p$-adic logarithm. For a multiplicative character $\chi$ modulo $p^m$, we define $c = c(\chi, a)$ to be the unique integer in the range
$0 < c \le \varphi(p^m)$ satisfying 
\begin{equation} \label{crel}
\chi\left(a^n\right)=e_{\varphi(p^m)}\left(cn\right) \quad \mbox{for every integer } n. 
\end{equation}
For any $f\in \mathbb{Z}[X]$, we define the $p$-adic order $\mbox{ord}_p(f)$ of $f$ as the largest integer $k$ such that $p^k$ divides all of the coefficients of $f$. If $f_1,f_2\in \mathbb{Z}[X]$, then we define the $p$-adic order of the rational function $f_1/f_2$ over $\mathbb{Z}/p^m\mathbb{Z}$ as 
$$
\mbox{ord}_p(f_1/f_2):=\mbox{ord}_p(f_1)-\mbox{ord}_p(f_2).
$$ 
We note that $\mbox{ord}_p(f)$ is well-defined. 
Now we set 
\begin{equation} \label{tdef}
t:=t_p(\chi,g,f) := \mbox{ord}_p(\mathcal{D}),
\end{equation}
where 
$$
\mathcal{D}:=Rgf' + cg'.
$$
We may assume that $\mbox{ord}_p(g) = 0$ for otherwise the exponential sum in question is empty. In this case, it was shown in \cite[Lemma 2.1]{Coch} that 
\begin{equation} \label{trel}
t = \min\left\{\mbox{ord}_p(f'),\mbox{ord}_p(cg')\right\}.
\end{equation}
We define the set of critical points $\mathcal{P}\subseteq\mathbb{F}_p$
associated with the sum $S(\chi,g,f,p^m)$ as
\begin{equation*}
\begin{split}
\mathcal{P} := & \left\{\alpha\in  \mathbb{F}_p : \mathcal{C}(\alpha)\equiv  0 \bmod{p}, \ f(\alpha) \mbox{ and } g(\alpha)\mbox{ are defined and } g(\alpha) \not\equiv  0 \bmod p\right\},
\end{split}
\end{equation*}
where 
$$
\mathcal{C}=p^{-t}\mathcal{D}=p^{-t}(Rgf'+cg').
$$
As remarked in \cite{Coch}, the above congruence $\mathcal{C}(\alpha)\equiv  0 \bmod{p}$ does not depend on the choice of the primitive root $a$. The multiplicity $\nu_p(\alpha)$ of a critical point $\alpha\in \mathcal{P}$ is defined as the order of vanishing of the rational function $\mathcal{C}$ at $\alpha$, i.e., the smallest non-negative integer $k$ such that the $k$-th derivative of the numerator of $\mathcal{C}$ (written as a reduced fraction) at $\alpha$ does not vanish modulo $p$.  
In the following, we write
\begin{equation} \label{Ssplitting}
\mathcal{S}(\chi,g,f,p^m) = \sum\limits_{\alpha=1}^{p} \mathcal{S}_{\alpha}(\chi,g,f,p^m)
\end{equation}
with
$$
S_{\alpha}(\chi,g,f,p^m) :=\sum\limits_{\substack{x=1\\ x\equiv \alpha\bmod{p}}}^{p^m}
\chi(g(x))e_{p^m}(f(x)).
$$
The following proposition provides a bound for this partial sum $\mathcal{S}_{\alpha}$. 

\begin{proposition}[Cochrane] \label{Cochprop} Suppose that $p$ is an odd prime, $m$ is a positive integer, $f$, $g$ are rational functions over $\mathbb{Z}$, not both constant,
and $\chi$ is any multiplicative character modulo $p^m$. Suppose that $t$, defined in \eqref{tdef}, satisfies $t\le m-2$. Put $\lambda:=(5/4)^5= 3.05...$.
If $\alpha$ is a critical point of multiplicity $\nu_{\alpha}\ge 1$, then
$$
|\mathcal{S}_{\alpha}(\chi,g,f,p^m)|\le \lambda_{\alpha}p^{t/(\nu_{\alpha}+1)}p^{m(1-1/(\nu_{\alpha}+1))},
$$
where $\lambda_{\alpha}:=\min\left\{\nu_{\alpha},\lambda\right\}$. For all $\alpha\in \mathbb{F}_p\setminus \mathcal{P}$ at which $f$ and $g$ are defined, we have
$$
\mathcal{S}_{\alpha}(\chi,g,f,p^m)=0.
$$ 
\end{proposition}

\begin{proof} This is contained in \cite[Theorem 1.2 and Theorem 1.1(i)]{Coch}. \end{proof}

A particular case of mixed exponential sums are Sali\'e sums, defined as 
\begin{equation} \label{Sali\'esum}
\mathcal{K}(a,b,c):=\sum_{n=0}^{c-1} \left(\frac{n}{c}\right)e_c(a\overline{n}+bn)
\end{equation}
for odd moduli $c$. Propositions \ref{primecase} and \ref{Cochprop} yield the following estimate, which we shall use in the course of this paper.

\begin{proposition} \label{Sali\'e}
Suppose that $c=p$ or $c=p^2$, where $p$ is an odd prime, and $(a,p)=1$. Then 
$$
\mathcal{K}(a,b,c)\ll c^{1/2}.
$$
\end{proposition}

It is well-known that Sali\'e sums can be evaluated explicitly (see \cite[Lemma 12.4]{IwKo}, for example), but we only need the above bound. 

Famously, in \cite{Hoo}, Hooley conjectured a bound for short Kloosterman sums, referred to as Hypothesis R$^{\ast}$, which predicts near square root cancellation. We propose the following analogue of this hypothesis for short Sali\'e sums which also extends hypothesis (C$^{\ast}$) in \cite{FI}. 

\begin{hypothesis}[Hypothesis R$^{\ast}$ for short Sali\'e sums] Let $a,b\in \mathbb{Z}$, $c\in \mathbb{Z}$ and $0\le x_1-x_2\le c$. Then 
\begin{equation*}
\sum\limits_{x_1<n\le x_2} \left(\frac{n}{c}\right)e_c(a\overline{n}+bn)\ll (x_2-x_1+1)^{1/2}c^{\varepsilon}(a,c)^{1/2}.
\end{equation*}
\end{hypothesis}

Finally, we will need the following well-known estimate for averages of greatest common divisors.

\begin{proposition} \label{gcdaverage} Let $r\in \mathbb{N}$ and $M\ge 1$. Then
\begin{equation*}
\sum\limits_{1\le m\le M} (r,m)\ll r^{\varepsilon}M.
\end{equation*}
\end{proposition}

\begin{proof} This is \cite[Lemma 2]{Baier} with $j=1$. \end{proof}

\section{Initial transformations} 
In this section, we carry out initial transformations of our bilinear exponential sums with modular square roots. We begin exactly like in \cite[subsection 3.2]{Baier} with applications of Weyl differencing and Poisson summation. For self-containedness, we copy the treatment in \cite[subsection 3.2]{Baier}, with some simplifications. 

\subsection{Weyl differencing and Poisson summation}
Write 
$$
\Sigma:=\Sigma_f(r,j,L,M,\boldsymbol{\alpha},\boldsymbol{\beta}),
$$ 
defined as in \eqref{Sigmadeff}. 
Using the Cauchy-Schwarz inequality, we have 
\begin{equation} \label{CauSch}
| \Sigma |^2
\ll
||\boldsymbol{\alpha}||_2^2 \sum\limits_{|l|\le L} \left|\sum\limits_{1\le m\le M} \beta_m e_r\left(l\sqrt{jm}\right)e(lf(m)) \right|^2.
\end{equation}
Let $\Phi:\mathbb{R}\rightarrow \mathbb{R}_{\ge 0}$ be a Schwartz class function such that $\Phi(x)=1$ if $-1\le x\le 1$. Then 
\begin{equation} \label{smoothing}
\begin{split} &
\sum\limits_{|l|\le L} \left|\sum\limits_{1\le m\le M} \beta_m e_r\left(l\sqrt{jm}\right)e(lf(m)) \right|^2 \\ \le &
\sum\limits_{l\in \mathbb{Z}} \Phi\left(\frac{l}{L}\right)\left|\sum\limits_{1\le m\le M} \beta_m e_r\left(l\sqrt{jm}\right)e(lf(m)) \right|^2 \\
 = &
\sum\limits_{l\in \mathbb{Z}} \Phi\left(\frac{l}{L}\right)\left|\sum\limits_{1\le m\le M} \beta_m e\left(l\left(\frac{\sqrt{jm}}{r}+f(m)\right)\right) \right|^2.
\end{split}
\end{equation}
Applying Proposition \ref{Weyldif} for $H\in \mathbb{N}\cap [1,M]$, we have 
\begin{equation} \label{Weyldifapp}
\begin{split}
& \left|\sum\limits_{1\le m\le M} \beta_m e\left(l\left(\frac{\sqrt{jm}}{r}+f(m)\right)\right) \right|^2\\ \ll & \frac{M}{H}\cdot \sum\limits_{\substack{1\le m_1,m_2\le M\\ |m_1-m_2|\le H}} \gamma(m_1,m_2)e\left(l\left(\frac{\sqrt{jm_1}-\sqrt{jm_2}}{r}+f(m_1)-f(m_2)\right)\right),
\end{split} 
\end{equation}
where $\gamma(m_1,m_2)$ is defined as in \eqref{cdef}. This remains true if $H\in [1,M]$ is real. We note that in contrast to \cite[subsection 3.2]{Baier}, we here have combined the diagonal and off-diagonal contributions into one single term.  
Combining
\eqref{CauSch}, \eqref{smoothing} and \eqref{Weyldifapp}, we deduce that
\begin{equation} \label{combisum}
\begin{split}
| \Sigma |^2
\ll & \frac{M}{H}\cdot ||\boldsymbol{\alpha}||_2^2\cdot  \sum\limits_{\substack{1\le m_1,m_2\le M\\ |m_1-m_2|\le H}} \gamma(m_1,m_2)
\sum\limits_{l\in \mathbb{Z}} \Phi\left(\frac{l}{L}\right)
e\left(l\left(\frac{\sqrt{jm_1}-\sqrt{jm_2}}{r}+f(m_1)-f(m_2)\right)\right).
\end{split}
\end{equation}
Now applying Proposition \ref{Poisson} to the sum over $l$ above, and using \eqref{nsumRHS}, we obtain
\begin{equation} \label{Poissonapp}
\begin{split}
& \sum\limits_{l\in \mathbb{Z}} \Phi\left(\frac{l}{L}\right)
e\left(l\left(\frac{\sqrt{jm_1}-\sqrt{jm_2}}{r}+f(m_1)-f(m_2)\right)\right)\\
\ll & L\mathbf{1}_{[0,r^{\varepsilon}L^{-1}]}\left(\left|\left|\frac{\sqrt{jm_1}-\sqrt{jm_2}}{r}+f(m_1)-f(m_2)\right|\right|\right)+r^{-2026}.
\end{split}
\end{equation}
If $1\le m_1,m_2\le M$ and $|m_1-m_2|\le H$, then by the mean value theorem from calculus, we have
$$
|f(m_1)-f(m_2)|\le FH
$$
under the conditions in Theorem \ref{bilinearboundnew}. Now we restrict $H$ to the range in \eqref{Hrange} and
note that $1/LF\ge 1$ by the condition $F\le 1/L$ in Theorem \ref{bilinearboundnew}. It follows that 
$$
|f(m_1)-f(m_2)|\le L^{-1}.
$$
As a consequence,
\begin{equation} \label{congrurewrite}
\begin{split}
& \mathbf{1}_{[0,r^{\varepsilon}L^{-1}]}\left(\left|\left|\frac{\sqrt{jm_1}-\sqrt{jm_2}}{r}+f(m_1)-f(m_2)\right|\right|\right)\\
\le & \begin{cases}  1 & \mbox{ if } \sqrt{jm_1}-\sqrt{jm_2}\equiv d \bmod{r} \mbox{ for some integer }d \mbox{ with } |d|\le D,\\ 0 & \mbox{ otherwise,}\end{cases}
\end{split}
\end{equation} 
where we set 
\begin{equation} \label{Ddefi}
D:=2r^{1+\varepsilon}L^{-1}.
\end{equation}
Combining \eqref{combisum}, \eqref{Poissonapp} and \eqref{congrurewrite}, and using $M\le r$ and 
$$
\gamma(m_1,m_2)\le ||\boldsymbol{\beta}||_{\infty}^2,
$$ 
we get 
\begin{equation} \label{beforeenergy}
| \Sigma |^2\ll
\frac{LM}{H}\cdot ||\boldsymbol{\alpha}||_2^2 ||\boldsymbol{\beta}||_{\infty}^2 \cdot \mathcal{A},
\end{equation} 
where 
$$
\mathcal{A}:=\sum\limits_{|d|\le D} \mathcal{A}(d)
$$
with 
\begin{equation*}
\mathcal{A}(d):=\sum\limits_{\substack{1\le m_1,m_2\le M\\ |m_1-m_2|\le H\\
\sqrt{jm_1}-\sqrt{jm_2}\equiv d\bmod{r}}} 1.
\end{equation*}

\subsection{Smoothing}
Recalling that $r^{\varepsilon}\le H\le M\le r^{1-\varepsilon}$, we may write 
\begin{equation} \label{maywrite}
\begin{split}
\mathcal{A}(d)=
&  \sum\limits_{\substack{k_1,k_2\bmod r\\ k_1-k_2\equiv d\bmod{r}\\ 0<\left\{\overline{j}k_i^2/r\right\}\le M/r\text{ for } i=1,2\\
\left|\left|\overline{j}(k_1^2-k_2^2)/r\right|\right|\le H/r}} 1\\
\le &  \sum\limits_{\substack{k_1,k_2\bmod r\\ k_1-k_2\equiv d\bmod{r}}} \mathbf{1}_{[0,M/r]}\left(\left|\left|\frac{\overline{j}k_2^2}{r}\right|\right|\right)
\mathbf{1}_{[0,H/r]}\left(\left|\left|\frac{\overline{j}\left(k_1^2-k_2^2\right)}{r}\right|\right|\right).
\end{split}
\end{equation}
Let $0<\nu\le 1$ and $\Phi:\mathbb{R}\rightarrow \mathbb{R}_{\ge 0}$ be a Schwartz class function satisfying $\Phi(x)=1$ if $-1\le x\le 1$. Then  
$$
\phi_{\nu}(y):=\sum\limits_{h\in \mathbb{Z}} \Phi\left(\frac{y+h}{\nu}\right) 
$$ 
defines a 1-periodic function on $\mathbb{R}$ such that 
\begin{equation} \label{envelop}
\mathbf{1}_{[0,\nu]}(||y||)\le \phi_{\nu}(y) 
\end{equation}
for all $y\in \mathbb{R}$. Applying Proposition \ref{Poisson}, we have 
\begin{equation} \label{afterP}
\phi_{\nu}(y):=\nu\sum\limits_{h\in \mathbb{Z}} \hat{\Phi}\left(h\nu\right)e(hy).
\end{equation}
Using \eqref{envelop}, we deduce from \eqref{maywrite} that
\begin{equation*}
\mathcal{A}(d)\le \sum\limits_{\substack{k_1,k_2\bmod r\\ k_1-k_2\equiv d\bmod{r}}} \phi_{M/r}
\left(\frac{\overline{j}k_2^2}{r}\right)
\phi_{H/r}\left(\frac{\overline{j}\left(k_1^2-k_2^2\right)}{r}\right),
\end{equation*}
which we simplify into
\begin{equation*}
\mathcal{A}(d)\le \sum\limits_{k\bmod r} \phi_{M/r}
\left(\frac{\overline{j}k^2}{r}\right)
\phi_{H/r}\left(\frac{\overline{j}\left(2dk+d^2\right)}{r}\right).
\end{equation*}
Combining this with \eqref{afterP}, and rearranging summations, we deduce that 
\begin{equation} \label{A1} 
\begin{split}
\mathcal{A}(d)\le & \frac{MH}{r^2}\cdot \sum\limits_{h_1\in \mathbb{Z}}\sum\limits_{h_2\in \mathbb{Z}} \hat\Phi\left(\frac{h_1M}{r}\right)\hat\Phi\left(\frac{h_2H}{r}\right) \sum\limits_{k\bmod r}e\left(\frac{\overline{j}(h_1k^2+2h_2dk+h_2d^2)}{r}\right)\\
= &  \frac{MH}{r^2}\cdot \sum\limits_{h_1\in \mathbb{Z}}\sum\limits_{h_2\in \mathbb{Z}} \hat\Phi\left(\frac{h_1M}{r}\right)\hat\Phi\left(\frac{h_2H}{r}\right) e_r\left(\overline{j}h_2d^2\right)G(\overline{j}h_1,2\overline{j}h_2d,r),
\end{split}
\end{equation}
where $G$ is the Gauss sum defined in \eqref{usualGauss}.

\subsection{The case of prime square moduli}
In this subsection, we assume that $r=p^2$ and $(j,p)=1$, where $p$ is an odd prime. In this case, we make the following useful observation. We note that if $(m,p^2)=p$, then $jm$ has no square root modulo $p^2$. Hence, if $1\le m\le M\le r^{1-\varepsilon}=p^{2-2\varepsilon}$, then a square root of $jm$ modulo $p^2$ exists if and only if $(m,p)=1$ and $jm$ is a quadratic residue modulo $p$. Consequently, we may include the coprimality condition $(k_2,r)=1$ in \eqref{maywrite}, i.e., we have 
\begin{equation} \label{addco}
\mathcal{A}(d)
\le \sum\limits_{\substack{k_1,k_2\bmod r\\ (k_2,r)=1\\ k_1-k_2\equiv d\bmod{r}}} \mathbf{1}_{[0,M/r]}\left(\left|\left|\frac{\overline{j}k_2^2}{r}\right|\right|\right)
\mathbf{1}_{[0,H/r]}\left(\left|\left|\frac{\overline{j}\left(k_1^2-k_2^2\right)}{r}\right|\right|\right).
\end{equation}
Now proceeding like in the previous subsection, we arrive at the inequality
\begin{equation} \label{A1res} 
\begin{split}
\mathcal{A}(d)\le \frac{MH}{p^4}\cdot \sum\limits_{h_1\in \mathbb{Z}}\sum\limits_{h_2\in \mathbb{Z}} \hat\Phi\left(\frac{h_1M}{p^2}\right)\hat\Phi\left(\frac{h_2H}{p^2}\right) e_{p^2}\left(\overline{j}h_2d^2\right) G^{\ast}\left(\overline{j}h_1,2\overline{j}h_2d,p^2\right)
\end{split}
\end{equation}
in place of \eqref{A1}, where $G^{\ast}$ is the restricted Gauss sum defined in \eqref{Gastdefi}. Bounding $\mathcal{A}(d)$ in terms of restricted instead of unrestricted Gauss sums is what allows us to handle prime square moduli.

\section{Evaluation of character sums}
In many estimates throughout this section, we will tacitly use that $r^{\varepsilon}\le L,M,H,D\le r^{1-\varepsilon}$ and $\Phi$ and $\hat\Phi$ decay rapidly. Without loss of generality, we will assume that $p$ is a prime greater than 3. 
\subsection{Prime moduli} 
It will be instructive to first handle prime moduli $r=p$ before extending these considerations later to general squarefree moduli. This gives a clear picture of the method. 
 
\subsubsection{Evaluation of quadratic Gauss sums} 
Suppose that $r=p$ is an odd prime. Evaluating the quadratic Gauss sum in the last line of \eqref{A1res} using Proposition \ref{Gaussprop}, we deduce that 
\begin{equation} \label{Ap}
\mathcal{A}=\sum\limits_{|d|\le D} \mathcal{A}(d)\le \sum\limits_{d\in \mathbb{Z}} \Phi\left(\frac{d}{D}\right)\mathcal{A}(d)=\mathcal{A}_1+\mathcal{A}_2+\mathcal{A}_3+\mathcal{A}_4
\end{equation}
with
\begin{equation} \label{A1p}
\mathcal{A}_1:= \frac{MH}{p^{3/2}}\cdot \epsilon_p\cdot  \sum\limits_{\substack{h_1\in \mathbb{Z}\\(h_1,p)=1}}\sum\limits_{h_2\in \mathbb{Z}} \hat\Phi\left(\frac{h_1M}{p}\right)\hat\Phi\left(\frac{h_2H}{p}\right)\left(\frac{jh_1}{p}\right)\sum\limits_{d\in \mathbb{Z}} \Phi\left(\frac{d}{D}\right)e_{p}\left(\overline{jh_1}h_2\left(h_1-h_2\right)d^2\right),
\end{equation}
\begin{equation*}
\mathcal{A}_2:= \frac{MH}{p}\cdot \sum\limits_{l_1\in \mathbb{Z}}\sum\limits_{l_2\in \mathbb{Z}} \hat\Phi\left(l_1M\right)\hat\Phi\left(l_2H\right)  \sum\limits_{d\in \mathbb{Z}} \Phi\left(\frac{d}{D}\right),
\end{equation*}
\begin{equation*}
\mathcal{A}_3:= \frac{MH}{p}\cdot \sum\limits_{l_1\in \mathbb{Z}}\sum\limits_{h_2\in \mathbb{Z}} \hat\Phi\left(l_1M\right)\hat\Phi\left(\frac{h_2H}{p}\right)  \sum\limits_{d_1\in \mathbb{Z}} \Phi\left(\frac{d_1p}{D}\right), 
\end{equation*}
\begin{equation*}
\mathcal{A}_4:= -\frac{MH}{p}\cdot \sum\limits_{l_1\in \mathbb{Z}}\sum\limits_{l_2\in \mathbb{Z}} \hat\Phi\left(l_1M\right)\hat\Phi\left(l_2H\right) \sum\limits_{d_1\in \mathbb{Z}} \Phi\left(\frac{d_1p}{D}\right),
\end{equation*}
where we have written
$$
l_1:=\frac{h_1}{p}, \quad l_2:=\frac{h_2}{p}\quad \mbox{and} \quad d_1:=\frac{d}{p}.
$$
We estimate the terms $\mathcal{A}_2,\mathcal{A}_3,\mathcal{A}_4$ trivially by 
\begin{equation} \label{A234}
\mathcal{A}_2\ll \frac{MHD}{p}, \quad 
\mathcal{A}_3\ll M, \quad
\mathcal{A}_4\ll \frac{MH}{p}.
\end{equation}  

Using Proposition \ref{Poisson}, we transform the $d$-sum in $\mathcal{A}_1$ into 
\begin{equation*}
\begin{split}
\sum\limits_{d\in \mathbb{Z}} \Phi\left(\frac{d}{D}\right)e_{p}\left(\overline{j}h_2\left(1-\overline{h_1}h_2\right)d^2\right)
= &  \frac{D}{p}\cdot \sum\limits_{u\in \mathbb{Z}} \hat{\Phi}\left(\frac{uD}{p}\right) G\left(\overline{jh_1}h_2\left(h_1-h_2\right),u,p\right).
\end{split}
\end{equation*}
Substituting this into \eqref{A1p}, evaluating the quadratic Gauss sum using Proposition \ref{Gaussprop}, and re-arranging summations, we deduce that 
\begin{equation}\label{A1psplit}
\mathcal{A}_1=\mathcal{A}_{1,1}+\mathcal{A}_{1,2}+\mathcal{A}_{1,3},
\end{equation}
where 
\begin{equation*}
\mathcal{A}_{1,1}:= \frac{MHD}{p^{3/2}}\cdot \epsilon_p\cdot  \sum\limits_{\substack{h_1\in \mathbb{Z}\\(h_1,p)=1}}\hat\Phi\left(\frac{h_1M}{p}\right) \left(\frac{jh_1}{p}\right) \sum\limits_{\substack{h_2\in \mathbb{Z}\\ p|h_2(h_1-h_2)}} \hat\Phi\left(\frac{h_2H}{p}\right) \sum\limits_{u_1\in \mathbb{Z}} \hat{\Phi}\left(u_1D\right),
\end{equation*}
\begin{equation} \label{A1p2}
\begin{split}
\mathcal{A}_{1,2}:= & \frac{MHD}{p^2}\cdot \epsilon_p^2\cdot  \sum\limits_{\substack{h_1\in \mathbb{Z}\\(h_1,p)=1}}\hat\Phi\left(\frac{h_1M}{p}\right) \sum\limits_{\substack{u\in \mathbb{Z}\\ (u,p)=1}} \hat{\Phi}\left(\frac{uD}{p}\right)\times\\ & \sum\limits_{\substack{h_2\in \mathbb{Z}\\ (h_2(h_1-h_2),p)=1}} \hat\Phi\left(\frac{h_2H}{p}\right)\left(\frac{h_2(h_1-h_2)}{p}\right)e_{p}\left(-jh_1u^2\overline{4h_2(h_1-h_2)}\right),
\end{split}
\end{equation}
\begin{equation} \label{A1p3} 
\begin{split}
\mathcal{A}_{1,3}:= & \frac{MHD}{p^2}\cdot \epsilon_p^2\cdot  \sum\limits_{\substack{h_1\in \mathbb{Z}\\(h_1,p)=1}}\hat\Phi\left(\frac{h_1M}{p}\right) \sum\limits_{u_1\in \mathbb{Z}} \hat{\Phi}\left(u_1D\right) \sum\limits_{\substack{h_2\in \mathbb{Z}\\ (h_2(h_1-h_2),p)=1}} \hat\Phi\left(\frac{h_2H}{p}\right)\left(\frac{h_2(h_1-h_2)}{p}\right),
\end{split}
\end{equation}
where we have written $u_1=u/p$. We estimate the term $\mathcal{A}_{1,1}$ trivially by 
\begin{equation} \label{A1,2}
\mathcal{A}_{1,1}\ll \frac{HD}{p^{1/2}}.
\end{equation}
The remainder of this subsection is dedicated to estimating $\mathcal{A}_{1,2}$ and $\mathcal{A}_{1,3}$. 

\subsubsection{Estimation of exponential sums} We use Proposition \ref{Poisson} to transform the $h_2$-sum on the right-hand side of \eqref{A1p2} into
\begin{equation} \label{h2transp}
\begin{split}
& \sum\limits_{\substack{h_2\in \mathbb{Z}\\ (h_2(h_1-h_2),p)=1}}\hat\Phi\left(\frac{h_2H}{p}\right) \left(\frac{h_2(h_1-h_2)}{p}\right)e_{p}\left(-jh_1u^2\overline{4h_2(h_1-h_2)}\right)\\
= & \frac{1}{H}\cdot \sum\limits_{v\in \mathbb{Z}} \Phi\left(\frac{v}{H}\right)\mathcal{S}(\chi,g,f,p),
\end{split}
\end{equation}
where $\chi=\left(\frac{\cdot}{p}\right)$ is the Legendre symbol and 
$$
f(x):=-\frac{jh_1u^2}{4g(x)}+vx,\quad g(x):=x(h_1-x)
$$
and $\mathcal{S}(\chi,g,f,p)$ is defined as in \eqref{mathS}. In the case when $p|u$, the sum in \eqref{h2transp} equals the innermost sum in \eqref{A1p3}. 

If $(h_1,p)=1$, then the discriminant $\mbox{disc}(g)$ of $g$ satisfies $(\mbox{disc}(g),p)=(h_1^2,p)=1$,
so the quadratic polynomial $g(x)$ is not a square of any linear polynomial over $\mathbb{F}_p$. Hence, in this case, condition (i) in Proposition \ref{primecase} is satisfied. Consequently, this proposition implies  
\begin{equation} \label{Sbo}
|\mathcal{S}(\chi,g,f,p)|\le 3p^{1/2}.
\end{equation}
We point out that this remains true if $p|u$. Combining \eqref{A1p2}, \eqref{h2transp}  and \eqref{Sbo}, we deduce that 
\begin{equation} \label{A1p2bo}
\mathcal{A}_{1,2}\ll \frac{MD}{p^{3/2}}\sum\limits_{\substack{h_1\in \mathbb{Z}\\(h_1,p)=1}}\left|\hat\Phi\left(\frac{h_1M}{p}\right)\right| \sum\limits_{\substack{u\in \mathbb{Z}\\ (u,p)=1}} \left|\hat{\Phi}\left(\frac{uD}{p}\right)\right|\sum\limits_{v\in \mathbb{Z}} \Phi\left(\frac{v}{H}\right)\ll Hp^{1/2}.
\end{equation}
Similarly, we obtain
\begin{equation} \label{A1p3bo}
\mathcal{A}_{1,3}\ll \frac{MD}{p^{3/2}}\sum\limits_{\substack{h_1\in \mathbb{Z}\\(h_1,p)=1}}\left|\hat\Phi\left(\frac{h_1M}{p}\right)\right| \sum\limits_{u_1\in \mathbb{Z}} \left|\hat{\Phi}\left(u_1D\right)\right|\sum\limits_{v\in \mathbb{Z}} \Phi\left(\frac{v}{H}\right)\ll \frac{HD}{p^{1/2}}.
\end{equation}

\subsubsection{Completion of the proof of Theorem \ref{bilinearboundnew}(i) for prime moduli}
Combining \eqref{Ap}, \eqref{A234}, \eqref{A1psplit}, \eqref{A1,2}, \eqref{A1p2bo} and \eqref{A1p3bo}, and using $H,D\ll p^{1-\varepsilon}$, we have the final estimate
\begin{equation} \label{so}
\mathcal{A}=\mathcal{A}_{1,2}+O\left(\frac{MHD}{p}+M+\frac{MH}{p}+\frac{HD}{p^{1/2}}\right)\ll  \frac{MHD}{p}+M+Hp^{1/2}.
\end{equation}
Using \eqref{Ddefi} with $r=p$, it follows that
\begin{equation*}
\mathcal{A}\ll \left(\frac{MH}{L}+M+Hp^{1/2}\right)p^{\varepsilon}.
\end{equation*}
Plugging this into \eqref{beforeenergy} and taking the square root yields
\begin{equation*}
| \Sigma |\ll \left(M+H^{-1/2}L^{1/2}M+L^{1/2}M^{1/2}p^{1/4}\right)||\boldsymbol{\alpha}||_2 ||\boldsymbol{\beta}||_{\infty}p^{\varepsilon},
\end{equation*} 
which implies the claimed bound \eqref{generalbilinearboundnew} in Theorem \ref{bilinearboundnew}(i) if $r=p$ is an odd prime.

\subsection{Prime square moduli} In this section, we prove Theorem \ref{bilinearboundnew}(ii). We will explain at the end of this subsection why we need to work with restricted instead of full Gauss sums to obtain a result which is useful for our purposes.
 
\subsubsection{Evaluation of quadratic Gauss sums} 
Evaluating the restricted quadratic Gauss sum in the last line of \eqref{A1res} using Proposition \ref{restrictedgausssums}, we deduce that 
\begin{equation} \label{A}
\mathcal{A}=\sum\limits_{|d|\le D} \mathcal{A}(d)\le \sum\limits_{d\in \mathbb{Z}} \Phi\left(\frac{d}{D}\right)\mathcal{A}(d)=\mathcal{A}_1+\mathcal{A}_2+\mathcal{A}_3+\mathcal{A}_4+\mathcal{A}_5+\mathcal{A}_6
\end{equation}
with
\begin{equation*}
\begin{split}
\mathcal{A}_1:= &\frac{MH}{p^3}\cdot \sum\limits_{\substack{h_1\in \mathbb{Z}\\(h_1,p)=1}}\sum\limits_{\substack{h_2\in \mathbb{Z}\\ (h_2,p)=1}} \hat\Phi\left(\frac{h_1M}{p^2}\right)\hat\Phi\left(\frac{h_2H}{p^2}\right)\sum\limits_{\substack{d\in \mathbb{Z}\\ (d,p)=1}} \Phi\left(\frac{d}{D}\right)e_{p^2}\left(\overline{jh_1}h_2\left(h_1-h_2\right)d^2\right),
\end{split}
\end{equation*}
\begin{equation*}
\begin{split}
\mathcal{A}_2:= & \frac{MH}{p^{5/2}}\cdot \epsilon_p\cdot \sum\limits_{\substack{l_1\in \mathbb{Z}\\(l_1,p)=1}}\sum\limits_{\substack{l_2\in \mathbb{Z}}} \hat\Phi\left(\frac{l_1M}{p}\right)\hat\Phi\left(\frac{l_2H}{p}\right)  \left(\frac{l_1}{p}\right) \sum\limits_{d\in \mathbb{Z}} \Phi\left(\frac{d}{D}\right)e_p\left(\overline{jh_1}h_2\left(h_1-h_2\right)d^2\right),
\end{split}
\end{equation*}
\begin{equation*}
\mathcal{A}_3:= -\frac{MH}{p^3}\cdot \sum\limits_{\substack{l_1\in \mathbb{Z}\\(l_1,p)=1}}\sum\limits_{\substack{l_2\in \mathbb{Z}}} \hat\Phi\left(\frac{l_1M}{p}\right)\hat\Phi\left(\frac{l_2H}{p}\right)  \sum\limits_{d\in \mathbb{Z}} \Phi\left(\frac{d}{D}\right), 
\end{equation*}
\begin{equation*}
\begin{split}
\mathcal{A}_4:= & \frac{MH}{p^{5/2}}\cdot \epsilon_p\cdot \sum\limits_{\substack{l_1\in \mathbb{Z}\\(l_1,p)=1}}\sum\limits_{\substack{h_2\in \mathbb{Z}}} \hat\Phi\left(\frac{l_1M}{p}\right)\hat\Phi\left(\frac{h_2H}{p^2}\right)  \left(\frac{jl_1}{p}\right)\sum\limits_{d_1\in \mathbb{Z}} \Phi\left(\frac{d_1p}{D}\right)e_p\left(\overline{jl_1}h_2^2d_1^2\right),
\end{split}
\end{equation*}
\begin{equation*}
\mathcal{A}_5:= -\frac{MH}{p^3}\cdot \sum\limits_{\substack{l_1\in \mathbb{Z}\\(l_1,p)=1}}\sum\limits_{\substack{h_2\in \mathbb{Z}}} \hat\Phi\left(\frac{l_1M}{p}\right)\hat\Phi\left(\frac{h_2H}{p^2}\right)  \sum\limits_{d_1\in \mathbb{Z}} \Phi\left(\frac{d_1p}{D}\right),
\end{equation*}
\begin{equation*}
\begin{split}
\mathcal{A}_6:= & \frac{MH}{p^4}\cdot \sum\limits_{l\in \mathbb{Z}}\sum\limits_{h_2\in \mathbb{Z}} \hat\Phi\left(lM\right)\hat\Phi\left(\frac{h_2H}{p^2}\right)
\sum\limits_{d\in \mathbb{Z}} \Phi\left(\frac{d}{D}\right)c_{p^2}(2jh_2d)\cdot e_{p^2}\left(\overline{j}h_2d^2\right),
\end{split}
\end{equation*}
where we have written
$$
l_1:=\frac{h_1}{p}, \quad l_2:=\frac{h_2}{p},\quad d_1:=\frac{d}{p} \quad \mbox{and} \quad l:=\frac{h_1}{p^2}.
$$

We estimate the terms $\mathcal{A}_3,...,\mathcal{A}_6$ trivially by 
$$
\mathcal{A}_3\ll \frac{MH}{p^3}\cdot \frac{p}{M}\cdot \left(1+\frac{p}{H}\right)D\ll \frac{HD}{p^2}+\frac{D}{p},
$$
$$
\mathcal{A}_4\ll \frac{MH}{p^{5/2}}\cdot \frac{p}{M}\cdot \frac{p^2}{H}\cdot \left(1+\frac{D}{p}\right)\ll
p^{1/2}+\frac{D}{p^{1/2}},
$$
$$
\mathcal{A}_5\ll \frac{MH}{p^{3}}\cdot \frac{p}{M}\cdot \frac{p^2}{H}\cdot \left(1+\frac{D}{p}\right)\ll
1+\frac{D}{p},
$$
\begin{equation*}
\begin{split}
\mathcal{A}_6\ll & \frac{MH}{p^4}\cdot \left(D+\frac{p^2}{H}\right)p^2+\frac{MH}{p^4}\cdot \sum\limits_{1\le k\le Dp^{2+\varepsilon}/H} \left(k,p^2\right)\tau(k)
\ll \left(\frac{MHD}{p^2}+M\right)p^{\varepsilon}.
\end{split}
\end{equation*}
In the estimation of $\mathcal{A}_6$ above, we have used \eqref{Ramanujan}. The first term $MH/p^4\cdot (D+p^2/H)p^2$ accounts for the contribution of $h_2d=0$, and the second term for the contribution of $h_2d\not=0$. To obtain the said second term, we have written $k=|h_2d|$. We have then estimated this term by  $\ll MD/p^{2-\varepsilon}$ using the divisor bound $\tau(k)\ll k^{\varepsilon}$ and Proposition \ref{gcdaverage}. This is dominated by the term $MHD/p^{2}\cdot p^{\varepsilon}$. We combine the above bounds into 
\begin{equation} \label{A3to6}
\mathcal{A}_3+\mathcal{A}_4+\mathcal{A}_5+\mathcal{A}_6\ll \left(\frac{MHD}{p^2}+M+\frac{D}{p^{1/2}}+p^{1/2}\right)p^{\varepsilon}. 
\end{equation}

Further, we divide the term $\mathcal{A}_2$ into 
\begin{equation} \label{A2split}
\mathcal{A}_2=\mathcal{A}_{2,0}+\mathcal{A}_{2,1},
\end{equation}
where 
\begin{equation*}
\begin{split}
\mathcal{A}_{2,0}:= & \frac{MH}{p^{5/2}}\cdot \epsilon_p\cdot \sum\limits_{\substack{l_1\in \mathbb{Z}\\(l_1,p)=1}}\sum\limits_{l_2\equiv 0 \text{ or } l_1\bmod{p}} \hat\Phi\left(\frac{l_1M}{p}\right)\hat\Phi\left(\frac{l_2H}{p}\right)  \left(\frac{l_1}{p}\right)\cdot \sum\limits_{d\in \mathbb{Z}} \Phi\left(\frac{d}{D}\right)
\end{split}
\end{equation*}
and 
\begin{equation*}
\begin{split}
\mathcal{A}_{2,1}:= & \frac{MH}{p^{5/2}}\cdot \epsilon_p\cdot \sum\limits_{\substack{l_1\in \mathbb{Z}\\(l_1,p)=1}}\sum\limits_{l_2\not\equiv 0,l_1\bmod{p}} \hat\Phi\left(\frac{l_1M}{p}\right)\hat\Phi\left(\frac{l_2H}{p}\right)  \left(\frac{l_1}{p}\right) \sum\limits_{d\in \mathbb{Z}} \Phi\left(\frac{d}{D}\right)e_p\left(\overline{jl_1}l_2\left(l_1-l_2\right)d^2\right).
\end{split}
\end{equation*}
Clearly,
\begin{equation} \label{A20}
\mathcal{A}_{2,0}\ll \frac{MH}{p^{5/2}}\cdot \frac{p}{M}\cdot \left(1+\frac{p}{H}\right)D\ll \frac{HD}{p^{3/2}}+\frac{D}{p^{1/2}}.
\end{equation}
If $(l_1,p)=1$ and $l_2\not\equiv 0,l_1\bmod{p}$, then using Proposition \ref{Poisson} and Proposition \ref{Gaussprop}(v), we transform the $d$-sum in $\mathcal{A}_{2,1}$ into
\begin{equation*} 
\begin{split}
\sum\limits_{d\in \mathbb{Z}} \Phi\left(\frac{d}{D}\right)e_p\left(\overline{jl_1}l_2\left(l_1-l_2\right)d^2\right)
= &  \frac{D}{p}\cdot \sum\limits_{u\in \mathbb{Z}} \hat{\Phi}\left(\frac{uD}{p}\right)G\left(\overline{jl_1}l_2\left(l_1-l_2\right),u,p\right)\\
= & \frac{D}{p^{1/2}}\cdot \epsilon_p\cdot \sum\limits_{u\in \mathbb{Z}} \hat{\Phi}\left(\frac{uD}{p}\right) \left(\frac{jl_1l_2\left(l_1-l_2\right)}{p}\right)e_p\left(-jl_1u^2\cdot \overline{4l_2\left(l_1-l_2\right)}\right)\\
\ll & \frac{D}{p^{1/2}}+p^{1/2},
\end{split}
\end{equation*}
where the final estimate is trivial. 
Consequently, 
\begin{equation} \label{A21}
\begin{split}
\mathcal{A}_{2,1}\ll \frac{MH}{p^{5/2}}\cdot \frac{p}{M}\cdot \frac{p}{H}\cdot \left(\frac{D}{p^{1/2}}+p^{1/2}\right)\ll 
\frac{D}{p}+1.
\end{split}
\end{equation}
Combining \eqref{A2split}, \eqref{A20} and \eqref{A21}, we obtain
\begin{equation} \label{A2esti}
\mathcal{A}_2\ll  \frac{HD}{p^{3/2}}+\frac{D}{p^{1/2}}+1.
\end{equation}

Further, we divide the term $\mathcal{A}_1$ into 
\begin{equation} \label{A1split}
\mathcal{A}_1:=\mathcal{A}_{1,0}+\mathcal{A}_{1,1}+\mathcal{A}_{1,2},
\end{equation}
where 
\begin{equation*}
\mathcal{A}_{1,0}:=\frac{MH}{p^3}\cdot \sum\limits_{\substack{h_1\in \mathbb{Z}\\(h_1,p)=1}}\sum\limits_{\substack{h_2\in \mathbb{Z}\\ h_2\equiv h_1\bmod{p^2}}} \hat\Phi\left(\frac{h_1M}{p^2}\right)\hat\Phi\left(\frac{h_2H}{p^2}\right)\sum\limits_{\substack{d\in \mathbb{Z}\\ (d,p)=1}} \Phi\left(\frac{d}{D}\right),
\end{equation*}
\begin{equation*}
\begin{split}
\mathcal{A}_{1,1}:= & \frac{MH}{p^3}\cdot \sum\limits_{\substack{h_1\in \mathbb{Z}\\(h_1,p)=1}}\sum\limits_{\substack{h_2\in \mathbb{Z}\\ \left(h_1-h_2,p^2\right)=p}} \hat\Phi\left(\frac{h_1M}{p^2}\right)\hat\Phi\left(\frac{h_2H}{p^2}\right)
\sum\limits_{\substack{d\in \mathbb{Z}\\ (d,p)=1}} \Phi\left(\frac{d}{D}\right)
e_p\left(\overline{jh_1}h_2\left(h_1-h_2\right)/p\cdot d^2\right),
\end{split}
\end{equation*}
\begin{equation*}
\begin{split}
\mathcal{A}_{1,2}:= & \frac{MH}{p^3}\cdot \sum\limits_{\substack{h_1\in \mathbb{Z}\\(h_1,p)=1}}\sum\limits_{\substack{h_2\in \mathbb{Z}\\ (h_2(h_1-h_2),p)=1}} \hat\Phi\left(\frac{h_1M}{p^2}\right)\hat\Phi\left(\frac{h_2H}{p^2}\right) \sum\limits_{\substack{d\in \mathbb{Z}\\ (d,p)=1}} \Phi\left(\frac{d}{D}\right)
e_{p^2}\left(\overline{jh_1}h_2\left(h_1-h_2\right)d^2\right).
\end{split}
\end{equation*}
We estimate the term $\mathcal{A}_{1,0}$ by
\begin{equation}\label{A10}
\mathcal{A}_{1,0}\ll \frac{MH}{p^3}\cdot \frac{p^2}{M}\cdot D\ll \frac{HD}{p}.
\end{equation}
Using Proposition \ref{Poisson}, Proposition \ref{Gaussprop}(v) and $(h_2(h_1-h_2)/p,p)=1$, we transform the $d$-sum in $\mathcal{A}_{1,1}$ into
\begin{equation*} 
\begin{split}
& \sum\limits_{\substack{d\in \mathbb{Z}\\ (d,p)=1}} \Phi\left(\frac{d}{D}\right)
e_p\left(\overline{jh_1}h_2\left(h_1-h_2\right)/p\cdot d^2\right)\\ 
= &  \frac{D}{p}\cdot \sum\limits_{u\in \mathbb{Z}} \hat{\Phi}\left(\frac{uD}{p}\right)G^{\ast}\left(\overline{jh_1}h_2\left(h_1-h_2\right)/p,u,p\right)\\
= & \frac{D}{p}\cdot \sum\limits_{u\in \mathbb{Z}} \hat{\Phi}\left(\frac{uD}{p}\right)\left(G\left(\overline{jh_1}h_2\left(h_1-h_2\right)/p,u,p\right)-1\right)\\
= &  \frac{D}{p^{1/2}}\cdot \epsilon_p\cdot \sum\limits_{u\in \mathbb{Z}} \hat{\Phi}\left(\frac{uD}{p}\right) \left(\frac{-jh_1h_2\left(h_1-h_2\right)/p}{p}\right) e_p\left(jh_1u^2\cdot \overline{4h_2\left(h_1-h_2\right)/p}\right)+
 O\left(\frac{D}{p}+1\right)\\
\ll & \frac{D}{p^{1/2}}+p^{1/2},
\end{split}
\end{equation*}
where the final estimate is trivial. Consequently, 
\begin{equation} \label{A11}
\mathcal{A}_{1,1}\ll  \frac{MH}{p^3}\cdot \frac{p^2}{M}\cdot \left(1+\frac{p}{H}\right)\left( \frac{D}{p^{1/2}}+p^{1/2}\right) \ll \frac{HD}{p^{3/2}}+\frac{H}{p^{1/2}}+\frac{D}{p^{1/2}}+p^{1/2}.
\end{equation}
We record that
\begin{equation} \label{A0app}
\mathcal{A}=\mathcal{A}_{1,2}+O\left(\left(\frac{MHD}{p^2}+\frac{HD}{p}+M+\frac{H}{p^{1/2}}+\frac{D}{p^{1/2}}+p^{1/2}\right)p^{\varepsilon}\right)
\end{equation}
using \eqref{A}, \eqref{A3to6}, \eqref{A2esti}, \eqref{A1split}, \eqref{A10} and \eqref{A11}. 

Finally, using  Proposition \ref{Poisson}, Proposition \ref{restrictedgausssums} and $(h_2(h_1-h_2),p)=1$, we transform the $d$-sum in $\mathcal{A}_{1,0}$ into 
\begin{equation*}
\begin{split}
\sum\limits_{\substack{d\in \mathbb{Z}\\ (d,p)=1}} \Phi\left(\frac{d}{D}\right)e_{p^2}\left(\overline{jh_1}h_2\left(h_1-h_2\right)d^2\right)
= &  \frac{D}{p^2}\cdot \sum\limits_{u\in \mathbb{Z}} \hat{\Phi}\left(\frac{uD}{p^2}\right) G^{\ast}\left(\overline{jh_1}h_2\left(h_1-h_2\right),u,p^2\right)\\
= &  \frac{D}{p}\cdot  \sum\limits_{\substack{u\in \mathbb{Z}\\ (u,p)=1}} \hat{\Phi}\left(\frac{uD}{p^2}\right) e_{p^2}\left(-jh_1u^2\overline{4h_2(h_1-h_2)}\right).
\end{split}
\end{equation*}
Hence, $\mathcal{A}_{1,2}$ takes the form
\begin{equation} \label{A12}
\begin{split}
\mathcal{A}_{1,2}
= & \frac{MHD}{p^4}\cdot \sum\limits_{\substack{h_1\in \mathbb{Z}\\(h_1,p)=1}} \hat\Phi\left(\frac{h_1M}{p^2}\right) \sum\limits_{\substack{u\in \mathbb{Z}\\ (u,p)=1}} \hat{\Phi}\left(\frac{uD}{p^2}\right)
\sum\limits_{\substack{h_2\in \mathbb{Z}\\ (h_2(h_1-h_2),p)=1}}\hat\Phi\left(\frac{h_2H}{p^2}\right) e_{p^2}\left(-jh_1u^2\overline{4h_2(h_1-h_2)}\right)
\end{split}
\end{equation}
after re-arranging summations.
The remainder of this subsection is dedicated to estimating the right-hand side of \eqref{A12}. 

\subsubsection{Estimation of exponential sums} We use Proposition \ref{Poisson} to transform the $h_2$-sum on the right-hand side of \eqref{A12} into
\begin{equation} \label{h2trans}
\begin{split}
& \sum\limits_{\substack{h_2\in \mathbb{Z}\\ (h_2(h_1-h_2),p)=1}}\hat\Phi\left(\frac{h_2H}{p^2}\right) e_{p^2}\left(-jh_1u^2\overline{4h_2(h_1-h_2)}\right)
= \frac{1}{H}\cdot \sum\limits_{v\in \mathbb{Z}} \Phi\left(\frac{v}{H}\right)\sum\limits_{\substack{n\bmod{p^2}\\ (n(h_1-n),p)=1}} e_{p^2}\left(f(n)\right)
\end{split}
\end{equation}
with
$$
f(x):=-\frac{jh_1u^2}{4x(h_1-x)}+vx.
$$
In the following, we set $h:=h_1$ for convenience. The derivative of $f(x)$, viewed as a rational function over the field $\mathbb{F}_p$, equals
$$
f'(x)=-\frac{jhu^2(2x-h)}{4x^2(h-x)^2}+v.
$$
We note that 
$$
t=t_p\left(f\right)=\mbox{ord}_p\left(f\right)=0
$$
since $(jhu,p)=1$. Now we write
$$
f(x)=\frac{A(2x-h)+4vx^2(h-x)^2}{4x^2(h-x)^2}=: \frac{g(x)}{4x^2(h-x)^2} \quad \mbox{with } A:=-jhu^2
$$
and investigate the multiplicities of zeros $\alpha\not=0,h$ of the polynomial $g(x)$ over $\mathbb{F}_p$ in what follows. Here we keep in mind that $(hA,p)=1$ and $p>3$.

If $p|v$, then the polynomial $g$ becomes linear over $\mathbb{F}_p$ and therefore has one simple root. In the following, assume that $(p,v)=1$. The discriminant of the polynomial $g$ is calculated to be 
$$
\mbox{disc}(g)=-256A^2v^2\left(27A^2+16v^2h^6\right). 
$$ 
Hence, $g$ has only simple roots in $\mathbb{F}_p$ unless
$$
27A^2+16v^2h^6\equiv 0\bmod{p},
$$
which is impossible unless $-3$ is a quadratic residue modulo $p$. If $\sqrt{-3}$ is a square root of $-3$ in $\mathbb{F}_p$, then the above congruence is equivalent to 
$$
v\equiv \pm \sqrt{-3}\cdot 3A\cdot \overline{4h^3}\bmod{p}. 
$$
The quartic polynomial $g$ does not have any root of multiplictity 4 in $\mathbb{F}_p$ since the second and third derivatives
$$
g''(x)=8v\left(6x^2-6hx+h^2\right) \quad \mbox{and}\quad g'''(x)=8v\left(12x-6h\right)
$$
do not share any roots in $\mathbb{F}_p$ if $p>3$. Now using Proposition \ref{Cochprop}, it follows that
$$
\sum\limits_{\substack{n\bmod{p^2}\\ (n(h-n),p)=1}} e_{p^2}\left(f(n)\right)\ll \begin{cases} p^{3/2} & \mbox{ if } v\equiv \pm \sqrt{-3} \cdot 3A\cdot \overline{4h^3}\bmod{p},\\
p & \mbox{ otherwise.} \end{cases} 
$$
Combining this with \eqref{h2trans}, we get
\begin{equation*} 
\begin{split}
& \sum\limits_{\substack{h_2\in \mathbb{Z}\\ (h_2(h_1-h_2),p)=1}}\hat\Phi\left(\frac{h_2H}{p^2}\right) e_{p^2}\left(-jh_1u^2\overline{4h_2(h_1-h_2)}\right)
\ll \frac{1}{H}\cdot \left(Hp+\left(\frac{H}{p}+1\right)p^{3/2}\right)\ll p+\frac{p^{3/2}}{H}.
\end{split}
\end{equation*}
Substituting this into \eqref{A12}, we obtain
\begin{equation} \label{A12esti}
\begin{split}
\mathcal{A}_{1,2}
\ll & \frac{MHD}{p^4}\cdot \sum\limits_{\substack{h_1\in \mathbb{Z}\\(h_1,p)=1}} \left|\hat\Phi\left(\frac{h_1M}{p^2}\right)\right| \sum\limits_{\substack{u\in \mathbb{Z}\\ (u,p)=1}} \left|\hat{\Phi}\left(\frac{uD}{p^2}\right)\right|\left(p+\frac{p^{3/2}}{H}\right)\ll Hp+p^{3/2}. 
\end{split}
\end{equation}
\begin{remark} \label{rem2} Above, we have made crucial use of the fact that multiple roots occur with a density $\ll 1/p$ as $v$ runs over the integers. This saves a lot if the modulus is $r=p^2$ but nothing substantial if the modulus is $r=p^n$ with a large power $n$. Thus, in potential future work along similar lines, a treatment of moduli $r$ containing large prime powers will require different techniques such as $p$-adic Weyl differencing.   \end{remark}

\subsubsection{Completion of the proof of Theorem \ref{bilinearboundnew}(ii)}
Combining \eqref{A0app} and \eqref{A12esti}, we have the final estimate
\begin{equation*}
\mathcal{A}\ll \left(\frac{MHD}{p^2}+\frac{HD}{p}+M+Hp+\frac{D}{p^{1/2}}+p^{3/2}\right)p^{\varepsilon}\ll \left(\frac{MHD}{p^2}+M+Hp+p^{3/2}\right)p^{\varepsilon}.
\end{equation*}
Using \eqref{Ddefi} with $r=p^2$, it follows that
\begin{equation*}
\mathcal{A}\ll \left(\frac{MH}{L}+M+Hp+p^{3/2}\right)p^{\varepsilon}.
\end{equation*}
Plugging this into \eqref{beforeenergy} and taking the square root yields
\begin{equation*}
\begin{split}
| \Sigma |\ll & \left(M+H^{-1/2}L^{1/2}M+L^{1/2}M^{1/2}p^{1/2}+H^{-1/2}L^{1/2}M^{1/2}p^{3/4}\right)||\boldsymbol{\alpha}||_2||\boldsymbol{\beta}||_{\infty}p^{\varepsilon},
\end{split}
\end{equation*} 
which implies the claimed bound \eqref{mainbound1} in Theorem \ref{bilinearboundnew}(ii) for $r=p^2$.  

\begin{remark} In the following, we explain why it is essential to add the coprimality condition $(k_2,r)=1$ in \eqref{addco}, resulting in a restricted Gauss sum in \eqref{A1res}. If we would estimate $\mathcal{A}$ for $r=p^2$ in the same way as for $r=p$, we would have a full Gauss sum $G\left(\overline{j}h_1,2\overline{j}h_2d,p^2\right)$ in place of $G^{\ast}\left(\overline{j}h_1,2\overline{j}h_2d,p^2\right)$ in \eqref{A1res}. If $(h_1,p)=1$ and $p|d$, this evaluates to 
$$
G\left(\overline{j}h_1,2\overline{j}h_2d,p^2\right)=p
$$
using Proposition \ref{Gaussprop}(v), resulting in a contribution of 
$$
\mathcal{A}_0:=\frac{MH}{p^4}\cdot \sum\limits_{\substack{h_1\in \mathbb{Z}\\(h_1,p)=1}}\sum\limits_{h_2\in \mathbb{Z}} \hat\Phi\left(\frac{h_1M}{p^2}\right)\hat\Phi\left(\frac{h_2H}{p^2}\right)\sum\limits_{\substack{d\in \mathbb{Z}\\ p|d}} \Phi\left(\frac{d}{D}\right)e_{p^2} \asymp \frac{MH}{p^4}\cdot \frac{p^2}{M}\cdot \frac{p^2}{H}\cdot p\left(1+\frac{D}{p}\right)= p+D
$$
to $\mathcal{A}$. 
The term $D$ is too large for our purposes since it would result in a term of size $H^{-1/2}M^{1/2}r^{1/2}$ in our estimate of $\Sigma_f$. This, in turn, would result in a term of size $H^{-1/2}L^{-1/2}M^{1/2}r^{1/2}$ in our estimate of the quantity $P(b/r+z)$ related to the large sieve with square moduli. Our choice of parameters in section \ref{lssection} would then just give the "trivial" bound $P(b/r+z)\ll Q^{1/2+\varepsilon}$.    
\end{remark}

\subsection{Consequences of Hooley's hypothesis R$^{\ast}$}
In this subsection, we prove Theorem \ref{bilinearboundnew}(iii). Both the crucial terms $\mathcal{A}_{1,2}$ in the treatments of prime and prime square moduli (see \eqref{A1p2} and \eqref{A12}) are of the same form
\begin{equation} \label{Bdefine}
\mathcal{A}_{1,2}= \frac{MHD}{r^2}\cdot \epsilon_r^2\cdot \sum\limits_{\substack{u\in \mathbb{Z}\\ (u,r)=1}} \hat{\Phi}\left(\frac{uD}{r}\right) \mathcal{B}(u)
\end{equation}
with
\begin{equation*}
\mathcal{B}(u):=\sum\limits_{\substack{h_1\in \mathbb{Z}\\(h_1,r)=1}}\sum\limits_{\substack{h_2\in \mathbb{Z}\\ (h_2(h_1-h_2),r)=1}}\hat\Phi\left(\frac{h_1M}{r}\right)  \hat\Phi\left(\frac{h_2H}{r}\right)\left(\frac{h_2(h_1-h_2)}{r}\right)e_{r}\left(-\frac{jh_1u^2}{4h_2(h_1-h_2)}\right),
\end{equation*}
where $r=p,p^2$, respectively. (Note here that $\epsilon_{p^2}=1$ and $\left(\frac{h_2(h_1-h_2)}{p^2}\right)=1$ if $(h_2(h_1-h_2),p)=1$.)
Using the factorization
\begin{equation*} 
\begin{split}
& \left(\frac{h_2(h_1-h_2)}{r}\right)e_r\left(-\frac{jh_1u^2}{4h_2(h_1-h_2)}\right)
= \left(\frac{h_2}{r}\right)e_r\left(-\frac{ju^2}{4h_2}\right) \left(\frac{h_1-h_2}{r}\right)e_r\left(-\frac{ju^2}{4(h_1-h_2)}\right),
\end{split}
\end{equation*}
we may write 
\begin{equation} \label{Burew}
\begin{split}
\mathcal{B}(u)= & \sum\limits_{\substack{h_2\in \mathbb{Z}\\ (h_2,r)=1}} \hat\Phi\left(\frac{h_2H}{r}\right) \left(\frac{h_2}{r}\right)e_r\left(-\frac{ju^2}{4h_2}\right) 
\sum\limits_{\substack{h_1\in \mathbb{Z}\\ (h_1(h_1-h_2),r)=1}} \hat\Phi\left(\frac{h_1M}{r}\right)
\left(\frac{h_1-h_2}{r}\right)e_r\left(-\frac{ju^2}{4(h_1-h_2)}\right)\\ 
= &  \sum\limits_{\substack{h_2\in \mathbb{Z}\\ (h_2,r)=1}} \hat\Phi\left(\frac{h_2H}{r}\right) \left(\frac{h_2}{r}\right)e_r\left(-\frac{ju^2}{4h_2}\right) 
\sum\limits_{\substack{h\in \mathbb{Z}\\ (h,r)=1\\ (h+h_2,r)=1}} \hat\Phi\left(\frac{(h+h_2)M}{r}\right)
\left(\frac{h}{r}\right)e_r\left(-\frac{ju^2}{4h}\right)
\end{split}
\end{equation}
by a change of variables $h:=h_1-h_2$. We note that the coprimality condition $(h,r)=1$ is taken care of by the Jacobi symbol and remove the coprimality condition $(h+h_2,r)=1$ using M\"obius inversion, getting 
\begin{equation*}
\begin{split}
& \sum\limits_{\substack{h\in \mathbb{Z}\\ (h,r)=1\\ (h+h_2,r)=1}} \hat\Phi\left(\frac{(h+h_2)M}{r}\right)
\left(\frac{h}{r}\right)e_r\left(-\frac{ju^2}{4h}\right)\\ 
= & \sum\limits_{h\in \mathbb{Z}} \hat\Phi\left(\frac{(h+h_2)M}{r}\right) \left(\frac{h}{r}\right)e_r\left(-\frac{ju^2}{4h}\right) +\sum\limits_{\substack{e|r\\ e>1}} \mu(e) \sum\limits_{\substack{h\in \mathbb{Z}\\ h\equiv -h_2\bmod{e}}} \hat\Phi\left(\frac{(h+h_2)M}{r}\right) \left(\frac{h}{r}\right)e_r\left(-\frac{ju^2}{4h}\right)\\
=&  \sum\limits_{h\in \mathbb{Z}} \hat\Phi\left(\frac{(h+h_2)M}{r}\right) \left(\frac{h}{r}\right)e_r\left(-\frac{ju^2}{4h}\right) + O\Bigg(\sum\limits_{\substack{e|r\\ e>1}} \left(1+\frac{r}{eM}\right)\Bigg).
\end{split}
\end{equation*}
In view of \eqref{Burew}, this implies 
\begin{equation}\label{Buapp}
\mathcal{B}(u)=  \tilde{\mathcal{B}}(u)+ O\left(\tau(r)\left(\frac{r}{H}+\max\limits_{\substack{e|r\\ e>1}} \frac{r^2}{eMH}\right)\right)
\end{equation}
with 
\begin{equation} \label{tildeB}
\tilde{\mathcal{B}}(u):=\sum\limits_{h_2\in \mathbb{Z}} \hat\Phi\left(\frac{h_2H}{r}\right) \left(\frac{h_2}{r}\right)e_r\left(\frac{ju^2}{4h_2}\right) 
\sum\limits_{h\in \mathbb{Z}} \hat\Phi\left(\frac{(h+h_2)M}{r}\right)
\left(\frac{h}{r}\right)e_r\left(\frac{ju^2}{4h}\right),
\end{equation}
where we note again that the coprimality condition $(h_2,r)=1$ is taken care of by the Jacobi symbol. 
For the innermost sum over $h$, Proposition \ref{Poisson} gives
\begin{equation*}
\begin{split}
\sum\limits_{h\in \mathbb{Z}} \hat\Phi\left(\frac{(h+h_2)M}{r}\right)
\left(\frac{h}{r}\right)e_r\left(\frac{ju^2}{4h}\right)
= & \frac{1}{M}\sum\limits_{k\in \mathbb{Z}} \Phi\left(\frac{k}{M}\right)e_r\left(kh_2\right)\mathcal{K}(\overline{4}ju^2,k,r),
\end{split}
\end{equation*}
where  $\mathcal{K}(\overline{4}ju^2,k,r)$ is the complete Sali\'e sum defined in \eqref{Sali\'esum}. 
Plugging this into \eqref{tildeB} and rearranging summations, we deduce that 
\begin{equation*}
\mathcal{\tilde{B}}(u)
=  \frac{1}{M} \sum\limits_{k\in \mathbb{Z}} \Phi\left(\frac{k}{M}\right)  \mathcal{K}(\overline{4}ju^2,k,r)\sum\limits_{h_2\in \mathbb{Z}} \hat\Phi\left(\frac{h_2H}{r}\right) \left(\frac{h_2}{r}\right)e_r\left(\frac{ju^2}{4h_2}+kh_2\right).
\end{equation*}
Now estimating the complete Sali\'e sum $\mathcal{K}(\overline{4}ju^2,k,r)$ using Proposition \ref{Sali\'e} and the incomplete Sali\'e sum over $h_2$ on the right-hand side of \eqref{tildeB} using partial summation and Hooley's Hypothesis R$^{\ast}$ for short Sali\'e sums, we obtain 
\begin{equation} \label{finalBu}
\tilde{\mathcal{B}}(u)\ll r^{1/2} \left(\frac{r}{H}\right)^{1/2}r^{\varepsilon}=\frac{r^{1+\varepsilon}}{H^{1/2}} \quad \mbox{ if } (u,r)=1.
\end{equation}
Combining \eqref{Bdefine}, \eqref{Buapp} and \eqref{finalBu}, we obtain
\begin{equation*}
\mathcal{A}_{1,2}\ll r^{\varepsilon}\left(MH^{1/2}+\max\limits_{\substack{e|r\\ e>1}} \frac{r}{e}\right)
\ll \begin{cases} MH^{1/2}p^{\varepsilon} & \mbox{ if } r=p,\\ \\
\left(MH^{1/2}+p\right)p^{\varepsilon} & \mbox{ if } r=p^2. \end{cases}
\end{equation*}
By the considerations in the previous subsections (in particular, \eqref{so} and \eqref{A0app}) and the definition of $D$ in \eqref{Ddefi}, it follows that
\begin{equation*}
\mathcal{A}\ll \begin{cases} \left(MHL^{-1}+Hp^{1/2}L^{-1}+MH^{1/2}\right)p^{\varepsilon} & \mbox{ if } r=p\\ \\ \left(MHL^{-1}+HpL^{-1}+Hp^{-1/2}+p^{3/2}L^{-1}+MH^{1/2}+p\right)p^{\varepsilon} & \mbox{ if } r=p^2.\end{cases}
\end{equation*} 
Plugging the above into \eqref{beforeenergy} and taking the square root yields
\begin{equation*}
\begin{split}
& \Sigma(r,j,L,M,\boldsymbol{\alpha},\boldsymbol{\beta},f)\ll ||\boldsymbol{\alpha}||_2 ||\boldsymbol{\beta}||_{\infty}p^{\varepsilon}\times\\ & \begin{cases}
(M+M^{1/2}p^{1/4}+H^{-1/4}L^{1/2}M)   & \mbox{ if } r=p\\ \\
(M+M^{1/2}p^{1/2}+L^{1/2}M^{1/2}p^{-1/4}+H^{-1/2}M^{1/2}p^{3/4}+\\
H^{-1/4}L^{1/2}M+H^{-1/2}L^{1/2}M^{1/2}p^{1/2}) & \mbox{ if } r=p^2,\end{cases}
\end{split}
\end{equation*} 
which implies \eqref{generalbilinearboundnewR} and \eqref{mainbound1R}, completing the proof of Theorem \ref{bilinearboundnew}(iii).

\subsection{Squarefree moduli} 
In this section, we extend our considerations for prime moduli to odd squarefree moduli, which complicates matters a lot but yields the same result. Hence, we will prove Theorem \ref{bilinearboundnew}(i).

\subsubsection{Evaluation of quadratic Gauss sums} 
Throughout the following, we assume that $r$ is odd. Suppose that $(h_1,r)=g$ and set 
\begin{equation*}
r_1:=\frac{r}{g}, \quad l_1:=\frac{h_1}{g}, \quad g_1:=(d,g),\quad g_2:=\frac{g}{g_1}, \quad d_1:=\frac{d}{g_1} \quad \mbox{and} \quad l_2=\frac{h_2}{g_2}. 
\end{equation*}
Then $(l_1,r_1)=1$ and $(d_1,g_2)=1$, and combining parts (i), (ii) and (v) of Proposition \ref{Gaussprop}, we have 
\begin{equation*} 
\begin{split} 
 e_r\left(\overline{j}h_2d^2\right)\cdot G(\overline{j}h_1,2h_2d,r)
= &\begin{cases} g\epsilon_{r_1}\cdot e_r\left(\overline{j}h_2d^2\right)\cdot \left(\frac{\overline{j}h_1/g}{r/g}\right)\cdot e_{r/g}\left(-\overline{jh_1/g}(h_2d/g)^2\right)\cdot \sqrt{\frac{r}{g}}& \mbox{if } g|h_2d \\ 0 &\mbox{otherwise}\end{cases}\\
= & 
\begin{cases} g\epsilon_{r_1}\cdot \left(\frac{jl_1}{r_1}\right)\cdot e_{r_1}\left(\overline{j}l_2\left(g_1-\overline{l_1}l_2\right)d_1^2\right)\cdot \sqrt{r_1} & \mbox{ if } g_2|h_2,\\ 0 & \mbox{ otherwise,} \end{cases}
\end{split}
\end{equation*}
where $\overline{l_1}l_1\equiv 1\bmod{r_1}$. 
Substituting this into \eqref{A1}, rearranging summations, and recalling that $r_1=r/g$, $g_2=g/g_1$ and $l_2=h_2/g_2$, it follows that 
\begin{equation*}
\mathcal{A}(d)\le \frac{MH}{r^{3/2}}\cdot \sum\limits_{g|r} g^{1/2}c_{r_1}  \sum\limits_{\substack{l_1,\in \mathbb{Z} \\
(l_1,r_1)=1}} \sum\limits_{l_2\in \mathbb{Z}} \hat\Phi\left(\frac{l_1M}{r_1}\right)\hat\Phi\left(\frac{l_2H}{g_1r_1}\right) \left(\frac{jl_1}{r_1}\right)e_{r_1}\left(\overline{jl_1}l_2\left(g_1l_1-l_2\right)d_1^2\right).
\end{equation*}
This implies
\begin{equation*} 
\begin{split}
\mathcal{A}= & \sum\limits_{|d|\le D} \mathcal{A}(d)\\
\le & \sum\limits_{d\in \mathbb{Z}} \Phi\left(\frac{d}{D}\right) \mathcal{A}(d)\\
\le & \frac{MH}{r^{3/2}}\cdot \sum\limits_{g|r} g^{1/2}c_{r_1} \sum\limits_{g_1|g}\ \sum\limits_{\substack{l_1\in \mathbb{Z}^2 \\
(l_1,r_1)=1}} \hat\Phi\left(\frac{l_1M}{r_1}\right)\sum\limits_{l_2\in \mathbb{Z}}  \hat\Phi\left(\frac{l_2H}{g_1r_1}\right) \left(\frac{jl_1}{r_1}\right)\times\\ &  \sum\limits_{\substack{d_1\in \mathbb{Z}\\ (d_1,g_2)=1}} \Phi\left(\frac{g_1d_1}{D}\right)  e_{r_1}\left(\overline{jl_1}l_2\left(g_1l_1-l_2\right)d_1^2\right).
\end{split}
\end{equation*}

We first bound the contribution of $g=r$ to the right-hand side above, which is the only case in which the term $l_1=0$ is present. In this case, we have $r_1=1$, and the said contribution is therefore dominated by 
$$
\ll \frac{MH}{r}\cdot \sum\limits_{g_1|r} \left(\frac{g_1}{H}+1\right)\left(\frac{D}{g_1}+1\right)\ll Mr^{\varepsilon}+\frac{MHD}{r}.
$$
It follows that
\begin{equation}\label{Aid0}
\mathcal{A}\ll Mr^{\varepsilon}+\frac{MHD}{r}+|\mathcal{A}_1|,
\end{equation}
where 
\begin{equation} \label{Aid}
\begin{split}
\mathcal{A}_1:= & \frac{MH}{r^{3/2}}\cdot \sum\limits_{\substack{g|r\\ g<r}} g^{1/2}c_{r_1} \sum\limits_{g_1|g}\ \sum\limits_{\substack{l_1\in \mathbb{Z}^2 \\
(l_1,r_1)=1}} \hat\Phi\left(\frac{l_1M}{r_1}\right)\sum\limits_{l_2\in \mathbb{Z}}  \hat\Phi\left(\frac{l_2H}{g_1r_1}\right) \left(\frac{jl_1}{r_1}\right)\times\\ &  \sum\limits_{\substack{d_1\in \mathbb{Z}\\ (d_1,g_2)=1}} \Phi\left(\frac{g_1d_1}{D}\right)  e_{r_1}\left(\overline{jl_1}l_2\left(g_1l_1-l_2\right)d_1^2\right).
\end{split}
\end{equation}
In the following, we assume that $g<r$ and thus $r_1>1$.

We remove the coprimality condition $(d_1,g_2)=1$ in the $d_1$-summation using M\"obius inversion, getting
\begin{equation} \label{Mobi}
\begin{split}
& \sum\limits_{\substack{d_1\in \mathbb{Z}\\ (d_1,g_2)=1}} \Phi\left(\frac{g_1d_1}{D}\right)  e_{r_1}\left(\overline{jl_1}l_2\left(g_1l_1-l_2\right)d_1^2\right)
= \sum\limits_{g_3|g_2} \mu(g_3)\sum\limits_{d_2\in \mathbb{Z}} \Phi\left(\frac{g_1g_3d_2}{D}\right)e_{r_1}\left(\overline{jl_1}g_3^2l_2\left(g_1l_1-l_2\right)d_2^2\right).
\end{split}
\end{equation}
Further, dividing the $d_2$-summation on the right-hand side of \eqref{Aid} into residue classes modulo $r_1$ and again 
using Proposition \ref{Poisson}, we obtain
\begin{equation} \label{nextGauss}
\begin{split}
& \sum\limits_{d_2\in \mathbb{Z}} \Phi\left(\frac{g_1g_3d_2}{D}\right)e_{r_1}\left(\overline{jl_1}g_3^2l_2\left(g_1l-1-l_2\right)d_2^2\right)\\
= & \frac{D}{g_1g_3r_1}\cdot  \sum\limits_{t\in \mathbb{Z}} \hat\Phi\left(\frac{tD}{g_1g_3r_1}\right) \sum\limits_{n\bmod{r_1}} e_{r_1}\left(\overline{jl_1}g_3^2l_2\left(g_1l_1-l_2\right)n^2+tn\right)\\
= & \frac{D}{g_1g_3r_1} \cdot \sum\limits_{t\in \mathbb{Z}} \hat\Phi\left(\frac{tD}{g_1g_3r_1}\right) G\left(\overline{jl_1}g_3^2l_2\left(g_1l_1-l_2\right),t,r_1\right).
\end{split}
\end{equation}

Now suppose that 
\begin{equation} \label{suppose}
\left(g_3^2l_2(g_1l_1-l_2),r_1\right)=f, \quad s:=\frac{t}{f}, \quad r_2:=\frac{r_1}{f}.
\end{equation}
Recalling that $(l_1,r_1)=1$, it then follows from parts (i), (ii) and (v) of Proposition \ref{Gaussprop} that
\begin{equation} \label{nextev}
\begin{split}
& G\left(\overline{jl_1}g_3^2l_2\left(g_1l_1-l_2\right),t,r_1\right)\\ = & \begin{cases} f\epsilon_{r_2}\left(\frac{jg_3^2l_1l_2\left(g_1l_1-l_2\right)/f}{r_2}\right)e_{r_2}\left(-jl_1s^2\overline{4g_3^2l_2(g_1l_1-l_2)/f}\right)\sqrt{r_2}&  \mbox{ if } f|t\\ 
0 & \mbox{ otherwise.}
\end{cases}
\end{split}
\end{equation}
Moreover, we have 
\begin{equation} \label{re}
\frac{D}{g_1g_3r_1}\cdot f\sqrt{r_2}=\frac{D\sqrt{f}}{g_1g_3\sqrt{r_1}}=\frac{D\sqrt{fg}}{g_1g_3\sqrt{r}}. 
\end{equation}
Combining \eqref{Aid}, \eqref{Mobi}, \eqref{nextGauss}, \eqref{suppose}, \eqref{nextev} and \eqref{re}, and re-arranging summations, we obtain
\begin{equation} \label{Aid2}
\begin{split}
\mathcal{A}_1 = & \frac{MHD}{r^2}\cdot \sum\limits_{\substack{g|r\\ g<r}} g \sum\limits_{f|r_1} f^{1/2}\epsilon_{r_1}\epsilon_{r_2} \sum\limits_{g_1|g} \frac{1}{g_1} \cdot \sum\limits_{g_3|g_2} \frac{\mu_2(g_3)}{g_3} \cdot
\sum\limits_{\substack{l_1\in \mathbb{Z} \\ (l_1,r_1)=1}} 
\hat\Phi\left(\frac{l_1M}{r_1}\right) \left(\frac{jl_1}{r_1}\right)\times\\ & 
\sum\limits_{s\in \mathbb{Z}} \hat\Phi\left(\frac{sD}{g_1g_3r_2}\right)
\mathcal{S}(g,f,g_1,g_3,l_1,s),
\end{split}
\end{equation} 
where 
\begin{equation*}
\begin{split}
\mathcal{S}(g,f,g_1,g_3,l_1,s)
:= & \sum\limits_{\substack{l_2\in \mathbb{Z} 
\\ \left(g_3^2l_2(g_1l_1-l_2),r_1\right)=f}}\hat\Phi\left(\frac{l_2H}{g_1r_1}\right)\left(\frac{jg_3^2l_1l_2\left(g_1l_1-l_2\right)/f}{r_2}\right)\times\\ & e_{r_2}\left(-jl_1s^2\overline{4g_3^2l_2(g_1l_1-l_2)/f}\right).
\end{split}
\end{equation*} 

\subsubsection{Reduction to complete exponential sums}  
In the following, we reduce the sum $\mathcal{S}(g,f,g_1,g_3,l_1,s)$ above to a sum of incomplete exponential sums, which we then complete using Poisson summation. We note that $(g_3,r_1)=(g_3,r_2)=(g_3,f)=1$ since $(g,r_1)=1$ by sqarefreeness of $r$. Hence, using M\"obius inversion, we obtain
\begin{equation} \label{taking}
\begin{split}
& \mathcal{S}(g,f,g_1,g_3,l_1,s)\\
= & \sum\limits_{f_1|f} \sum\limits_{\substack{l_2\in \mathbb{Z} 
\\ \left(l_2,r_1\right)=f_1}}\hat\Phi\left(\frac{l_2H}{g_1r_1}\right)\left(\frac{jl_1l_2\left(g_1l_1-l_2\right)/f}{r_2}\right)e_{r_2}\left(-jl_1s^2\overline{4g_3^2l_2(g_1l_1-l_2)/f}\right)\\
= & \sum\limits_{f_1|f} \sum\limits_{f_1|f_2|f} \mu\left(\frac{f_2}{f_1}\right)\sum\limits_{\substack{l_2\in \mathbb{Z} 
\\ f_2|\left(l_2,r_1\right)}}\hat\Phi\left(\frac{l_2H}{g_1r_1}\right)\left(\frac{jl_1l_2\left(g_1l_1-l_2\right)/f}{r_2}\right)e_{r_2}\left(-jl_1s^2\overline{4g_3^2l_2(g_1l_1-l_2)/f}\right),
\end{split}
\end{equation} 
where have we taken into account that the summation condition $\left(g_3^2l_2(g_1l_1-l_2),r_1\right)=f$ in the $l_2$-summation in the first line can be omitted since it is taken care of by the Jacobi symbol. Now if $f_1|f_2|f$, we write 
$$
f_3:=\frac{f}{f_2}.
$$ 
Then it follows from $f_2|\left(l_2,r_1\right)$ and $\left(l_2(g_1l_1-l_2),r_1\right)=f$ that
$$
f_3|(g_1l_1-l_2).
$$
We further write 
$$
l_3:=\frac{l_2}{f_2}.
$$
Then $f_3|(g_1l_1-l_2)$ is equivalent to the congruence
$$
f_2l_3\equiv g_1l_1\bmod{f_3}.
$$
Note that $(l_1,f_3)=1$ since $(l_1,r_1)=1$ and $(f_2,f_3)=1$ since $f$ is squarefree. Hence, 
$$
l_3\equiv f_4g_1l_1\bmod{f_3} \quad \mbox{with } f_2f_4\equiv 1\bmod{f_3},
$$
which is equivalent to 
$$
l_3=f_4g_1l_1+f_3l_4\quad \mbox{for some } l_4\in \mathbb{Z}.
$$
We further write
$$
f_5:=\frac{1-f_2f_4}{f_3}.
$$
Then it follows that 
\begin{equation*}
\begin{split}
\frac{l_2(g_1l_1-l_2)}{f}= & \frac{l_3(g_1l_1-f_2l_3)}{f_3}
= \frac{(f_4g_1l_1+f_3l_4)(g_1l_1-f_2f_4g_1l_1-f_2f_3l_4)}{f_3}\\
= & (f_4g_1l_1+f_3l_4)(f_5g_1l_1-f_2l_4).
\end{split}
\end{equation*}
Hence, using \eqref{taking}, $\mathcal{S}(g,f,g_1,g_3,l_1,s)$ takes the form
\begin{equation} \label{Expst}
\mathcal{S}(g,f,g_1,g_3,l_1,s)
= \sum\limits_{f_1|f} \sum\limits_{\substack{f_3|f/f_1\\ (f/f_3,f_3)|g_1}} \mu\left(\frac{f}{f_1f_3}\right)\mathcal{S}(g,g_1,g_3,l_1,s,f_1,f_3)
\end{equation}
with 
\begin{equation} \label{beforedividing}
\begin{split}
\mathcal{S}(g,g_1,g_3,l_1,s,f_1,f_3) := &
\sum\limits_{l_4\in \mathbb{Z}}
\hat\Phi\left(\frac{f_4g_1l_1H}{g_1f_3r_2}+\frac{l_4H}{g_1r_2}\right)\cdot  \left(\frac{jl_1(f_4g_1l_1+f_3l_4)(f_5g_1l_1-f_2l_4)}{r_2}\right)\times\\ & e_{r_2}\left(-jl_1s^2\overline{4g_3^2(f_4g_1l_1+f_3l_4)(f_5g_1l_1-f_2l_4)}\right),
\end{split}
\end{equation} 
where we have transformed the argument inside the weight function $\hat\Phi$ in the form
$$
\frac{l_2H}{g_1fr_2}=\frac{l_3H}{g_1f_3r_2}=\frac{(f_4g_1l_1+l_4f_3)H}{g_1f_3r_2}=\frac{f_4g_1l_1H}{g_1f_3r_2}+\frac{l_4H}{g_1r_2}.
$$

Dividing the sum over $l_4$ into residue classes modulo $r_2$, we get
\begin{equation}\label{dividing}
\begin{split}
\mathcal{S}(g,g_1,g_3,l_1,s,f_1,f_3)= &
\sum\limits_{n\bmod{r_2}} \left(\frac{jl_1(f_4g_1l_1+f_3n)(f_5g_1l_1-f_2n)}{r_2}\right)\times\\ & e_{r_2}\left(-jl_1s^2\overline{4g_3^2(f_4g_1l_1+f_3n)(f_5g_1l_1-f_2n)}\right) \times\\ &  
\sum\limits_{u\in \mathbb{Z}}
\hat\Phi\left(\frac{f_4g_1l_1H}{g_1f_3r_2}+\frac{(ur_2+n)H}{g_1r_2}\right).
\end{split}
\end{equation} 
Now Proposition \ref{Poisson} gives 
\begin{equation}\label{Poi}
\begin{split}
\sum\limits_{u\in \mathbb{Z}}
\hat\Phi\left(\frac{f_4g_1l_1H}{g_1f_3r_2}+\frac{(ur_2+n)H}{g_1r_2}\right)= & \sum\limits_{u\in \mathbb{Z}}
\hat\Phi\left(\frac{f_4g_1l_1H}{g_1f_3r_2}+\frac{nH}{g_1r_2}+\frac{uH}{g_1}\right)\\
= & 
\frac{g_1}{H}\cdot \sum\limits_{v\in \mathbb{Z}} \Phi\left(-\frac{vg_1}{H}\right) e\left(\frac{vf_4g_1l_1}{f_3r_2}+\frac{vn}{r_2}\right).
\end{split}
\end{equation}
Combining \eqref{dividing} and \eqref{Poi}, and re-arranging summations, we deduce that 
\begin{equation}\label{endform}
\begin{split}
\mathcal{S}(g,g_1,g_3,l_1,s,f_1,f_3)= & \frac{g_1}{H}\cdot \sum\limits_{v\in \mathbb{Z}} \Phi\left(-\frac{vg_1}{H}\right) e\left(\frac{vf_4g_1l_1}{g_1f_3r_2}\right) \mathcal{E}_j(g_3^2,g_1,l_1,s^2,f_2,f_3,f_4,f_5,v;r_2),
\end{split}
\end{equation}
where 
\begin{equation*}
\begin{split}
 \mathcal{E}_j(g_3^2,g_1,l_1,a,f_2,f_3,f_4,f_5,v;r_2):= & \sum\limits_{n\bmod{r_2}} \left(\frac{jl_1(f_4g_1l_1+f_3n)(f_5g_1l_1-f_2n)}{r_2}\right)\times\\ & e_{r_2}\left(-jl_1a\overline{4g_3^2(f_4g_1l_1+f_3n)(f_5g_1l_1-f_2n)}+vn\right).
\end{split}
\end{equation*} 

\subsubsection{Multiplicativity of our exponential sums}
The following lemma asserts that the above complete exponential sum is multiplicative in the modulus.

\begin{lemma} \label{multi}
For any odd $q_1,q_2\in \mathbb{N}$ satisfying $(q_1,q_2)=1$, we have 
\begin{equation*}
\begin{split}
& \mathcal{E}_j(g_3^2,g_1,l_1,a,f_2,f_3,f_4,f_5,v;q_1q_2)\\
= & \mathcal{E}_j(g_3^2,g_1,l_1,a\overline{q_2},f_2q_2,f_3q_2,f_4,f_5,v;q_1)\mathcal{E}_j(g_3^2,g_1,l_1,a\overline{q_1},f_2q_1,f_3q_1,f_4,f_5,v;q_2).
\end{split}
\end{equation*}
\end{lemma}
\begin{proof} Since $(q_1,q_2)=1$, if $n_1$ ranges over all residue classes modulo $q_1$ and $n_2$ ranges over all residue classes modulo $q_2$, then $n_1q_2+n_2q_1$ ranges over all residue classes modulo $q_1q_2$. Consequently, we have 
\begin{equation*}
\begin{split}
& \mathcal{E}_j(g_3^2,g_1,l_1,a,f_2,f_3,f_4,f_5,v;q_1q_2)\\
= \sum\limits_{n\bmod q_1q_2} & \left(\frac{jl_1(f_4g_1l_1+f_3n)(f_5g_1l_1-f_2n)}{q_1q_2}\right) \times\\ & e_{q_1q_2}\left(-jl_1a\overline{4g_3^2(f_4g_1l_1+f_3n)(f_5g_1l_1-f_2n)}+vn\right)\\
= \sum\limits_{n\bmod q_1q_2} & \left(\frac{jl_1(f_4g_1l_1+f_3n)(f_5g_1l_1-f_2n)}{q_1}\right) \left(\frac{jg_3^2l_1(f_4g_1l_1+f_3n)(f_5g_1l_1-f_2n)}{q_2}\right)  \times\\ & 
e_{q_1q_2}\left(-jl_1a\overline{4g_3^2(f_4g_1l_1+f_3n)(f_5g_1l_1-f_2n)}+vn\right)\\
= \sum\limits_{n_1\bmod{q_1}} & \sum\limits_{n_2\bmod{q_2}} \left(\frac{jl_1(f_4g_1l_1+f_3q_2n_1)(f_5g_1l_1-f_2q_2n_1)}{q_1}\right)\times\\ & \left(\frac{jl_1(f_4g_1l_1+f_3q_1n_2)(f_5g_1l_1-f_2q_1n_2)}{q_2}\right) \times\\ & e_{q_1q_2}\left(-jl_1a\overline{4g_3^2(f_4g_1l_1+f_3(n_1q_2+n_2q_1))(f_5g_1l_1-f_2(n_1q_2+n_2q_1))}\right)\times\\ & e_{q_1q_2}\left(v(n_1q_2+n_2q_1)\right).
\end{split}
\end{equation*}
It is easily seen that 
\begin{equation*}
\begin{split}
& e_{q_1q_2}\left(-jl_1a\overline{4g_3^2(f_4g_1l_1+f_3(n_1q_2+n_2q_1))(f_5g_1l_1-f_2(n_1q_2+n_2q_1))}\right)e_{q_1q_2}\left(v(n_1q_2+n_2q_1)\right)\\
= & e_{q_1}\left(-jl_1a\overline{4g_3^2q_2(f_4g_1l_1+f_3q_2n_1)(f_5g_1l_1-f_2q_2n_1)}+vn_1\right)\times\\ & e_{q_2}\left(-jl_1a\overline{4g_3^2q_1(f_4g_1l_1+f_3q_1n_2)(f_5g_1l_1-f_2q_1n_2)}+vn_2\right).
\end{split}
\end{equation*}
Hence, $\mathcal{E}_j(g_3^2,g_1,l_1,a,f_2,f_3,f_4,f_5,v;q_1q_2)$ simplifies into 
\begin{equation*}
\begin{split}
& \mathcal{E}_j(g_3^2,g_1,l_1,a,f_2,f_3,f_4,f_5,v;q_1q_2)\\
= & \Bigg(\sum\limits_{n_1\bmod{q_1}}  \left(\frac{jl_1(f_4g_1l_1+f_3q_2n_1)(f_5g_1l_1-f_2q_2n_1)}{q_1}\right)\times\\ & \quad\quad  e_{q_1}\left(-jl_1a\overline{4g_3^2q_2(f_4g_1l_1+f_3q_2n_1)(f_5g_1l_1-f_2q_2n_1)}+vn_1\right)\Bigg)\times \\ & \Bigg(\sum\limits_{n_2\bmod{q_2}}  \left(\frac{jl_1(f_4g_1l_1+f_3q_1n_2)(f_5g_1l_1-f_2q_1n_2)}{q_2}\right)\times\\ & \quad\quad   e_{q_2}\left(-jl_1a\overline{4g_3^2q_1(f_4g_1l_1+f_3q_1n_2)(f_5g_1l_1-f_2q_1n_2)}+vn_2\right)\Bigg).\\
= & \mathcal{E}_j(g_3^2,g_1,l_1,a\overline{q_2},f_2q_2,f_3q_2,f_4,f_5,v;q_1) \mathcal{E}_j(g_3^2,g_1,l_1,a\overline{q_1},f_2q_1,f_3q_1,f_4,f_5,v;q_2).
\end{split}
\end{equation*}
This completes the proof. \end{proof}

This allows us to decompose $\mathcal{E}_j(g_3^2,g_1,l_1,a,f_2,f_3,f_4,f_5,v;r_2)$ into a product of complete exponential sums modulo prime powers. Indeed, Lemma \ref{multi} implies the following.

\begin{corollary} \label{product}
Suppose $r_2$ is squarefree and 
$$
r_2=\prod\limits_{i=1}^w p_i
$$ 
is the prime factorization of $r_2\in \mathbb{N}$. Then 
$$
\mathcal{E}_j(g_3^2,g_1,l_1,a,f_2,f_3,f_4,f_5,v;r_2)=\prod\limits_{i=1}^w \mathcal{E}_j(g_3^2,g_1,l_1,a\overline{q_i},f_2q_i,f_3q_i,f_4,f_5,v;p_i),
$$
where
$$
q_i:=\frac{r_2}{p_i} \quad \mbox{for } i=1,...,w.
$$
\end{corollary}

Now it suffices to estimate the exponential sums $\mathcal{E}_j(g_3^2,g_1,l_1,a\overline{q_i},f_2q_i,f_3q_i,f_4,f_5,v;p_i)$ to prime power moduli $p_i|r_2$ for $i=1,...,w$. We recall that $r_2$ and hence $p_i$ is odd for $i=1,...,w$. 

\subsubsection{Estimation of our exponential sums} Using the notations in section \ref{preli}, we may write
$$
\mathcal{E}_j(g_3^2,g_1,l_1,a\overline{q_i},f_2q_i,f_3q_i,f_4,f_5,v;p_i)=\mathcal{S}(\chi, g, f, p),
$$
where $p:=p_i$, $\chi=\left(\frac{\cdot}{p}\right)$, the Legendre symbol, and
$$
g(x):=\frac{jl_1}{F(x)} \quad \mbox{and} \quad  
f(x):=-\frac{jl_1a}{4g_3^2q_iF(x)}+vx
$$
with 
\begin{equation*}
\begin{split}
F(x):= & (f_4g_1l_1+f_3q_ix)(f_5g_1l_1-f_2q_ix)\\ 
= &
f_4f_5g_1^2l_1^2+g_1l_1q_i(f_3f_5-f_2f_4)x-q_i^2f_2f_3x^2\\
= & f_4f_5g_1^2l_1^2+g_1l_1q_i(1-2f_2f_4)x-q_i^2f_2f_3x^2.
\end{split}
\end{equation*}
We calculate the discriminant of $F$ to be 
\begin{equation*}
\begin{split}
\mbox{disc}(F)= & (g_1l_1q_i)^2\left((1-2f_2f_4)^2+4f_2f_4f_3f_5\right)\\
= & (g_1l_1q_i)^2\left((1-2f_2f_4)^2+4f_2f_4(1-f_2f_4)\right)\\
= & (g_1l_1q_i)^2,
\end{split}
\end{equation*}
which is coprime to $p$. (Note here that $(g_1,r_2)=1$ since $(g,r_1)=1$, $(l_1,p)=1$ since $(l_1,r_2)=1$, and $(q_i,p)=(q_i,p_i)=1$ by definition of $q_i$.)  
Therefore, the quadratic polynomial $F(x)$ is not a square of any linear polynomial over $\mathbb{F}_{p}$ and hence, condition (i) in Proposition \ref{primecase} is satisfied. Consequently, this proposition implies  
\begin{equation} \label{groundcase}
|\mathcal{S}(\chi,g,f,p)|\le 3p^{1/2}.
\end{equation}
We point out that this bound remains valid if $(v,a,p)=p$. 
Combining Corollary \ref{product} and \eqref{groundcase}, we deduce the following.

\begin{lemma} \label{expsumestimate}
We have 
$$
\mathcal{E}_j(g_3^2,g_1,l_1,a,f_2,f_3,f_4,f_5,v;r_2)\ll r_2^{1/2+\varepsilon}.
$$
\end{lemma}

Now we use Lemma \ref{expsumestimate} to estimate the sum $\mathcal{S}(g,g_1,g_3,l_1,s,f_1,f_3)$ in \eqref{endform} by
\begin{equation*}
\mathcal{S}(g,g_1,g_3,l_1,s,f_1,f_3)\ll \frac{g_1}{H}\cdot \sum\limits_{v\in \mathbb{Z}} \Phi\left(-\frac{vg_1}{H}\right) r_2^{1/2+\varepsilon}.
\end{equation*}
Plugging this into \eqref{Expst} and the result into \eqref{Aid2}, and using $r_1=r/g$ and $r_2=r_1/f=r/(fg)$, we get
\begin{equation*}
\begin{split}
\mathcal{A}_1
\ll & \frac{MD}{r^2}\cdot \sum\limits_{\substack{g|r\\ g<r}} g^{1/2} \sum\limits_{f|r/g} \sum\limits_{g_1|g} \sum\limits_{g_3|g/g_1} \frac{1}{g_3} 
\sum\limits_{\substack{l_1\in \mathbb{Z} \\
(l_1,r/g)=1}} \left|\hat\Phi\left(\frac{l_1gM}{r}\right)\right| \cdot 
\sum\limits_{s\in \mathbb{Z}} \left|\hat\Phi\left(\frac{sfgD}{g_1g_3r}\right)\right|\times\\ &
\sum\limits_{f_1|f} \sum\limits_{f_3|f/f_1} \sum\limits_{v\in \mathbb{Z}} \Phi\left(-\frac{vg_1}{H}\right) r^{1/2+2\varepsilon}\\
\ll & \frac{MD}{r^2}\cdot \sum\limits_{g|r} g^{1/2} \sum\limits_{f|r/g} \sum\limits_{g_1|g} \sum\limits_{g_3|g/g_1} \frac{1}{g_3}\cdot \frac{r}{gM}\cdot \left(1+\frac{g_1g_3r}{fgD}\right)\left(1+\frac{H}{g_1}\right)r^{1/2+2\varepsilon}\\
\ll & Hr^{1/2+3\varepsilon}.
\end{split}
\end{equation*}
Combining this with \eqref{Aid0} and using the definition of $D$ in \eqref{Ddefi}, it follows that 
$$
\mathcal{A} \ll \left(M+\frac{MH}{L}+Hr^{1/2}\right)r^{\varepsilon}.
$$
This, in conjunction with \eqref{beforeenergy}, gives 
\begin{equation*}
\begin{split}
| \Sigma |\ll &
\left(H^{-1/2}L^{1/2}M+M+L^{1/2}M^{1/2}r^{1/4}\right)||\boldsymbol{\alpha}||_2 ||\boldsymbol{\beta}||_{\infty}r^{\varepsilon},
\end{split}
\end{equation*} 
which completes the proof of Theorem \ref{bilinearboundnew}(i).

\section{Large sieve with square moduli} \label{lssection} 
In this section, we give a detailed proof of Theorem \ref{extendedrange}(i) for case of squarefree moduli $r$. Parts (ii) and (iii) of Theorem \ref{extendedrange} can be established in a similar way. Throughout the following, we will set 
\begin{equation*} 
x:=\frac{b}{r}+z.
\end{equation*}

\subsection{Estimation of $P(x)$}
We appeal to Theorem \ref{bilinearboundnew} to estimate the right-hand side of \eqref{specialPx}. 
Using the rapid decay of $W$, we may cut off the summation over $l$ in \eqref{specialPx} at $|l|\le Q^{1+\varepsilon}r/\delta$ at the cost of a negligible error of size $O(Q^{-2026})$. Since $V$ is supposed to have compact support in $\mathbb{R}_{>0}$, the summation over $m$ can be restricted to $C_0Q^{1/2-\gamma}\le m\le C_1Q^{1/2-\gamma}$ for suitable constants $C_1>C_0>0$. So setting
\begin{equation} \label{Ldef}
L:=\frac{Q^{1+\varepsilon}r}{\delta}
\end{equation}
and 
\begin{equation} \label{Mdef}
\quad M_0:=C_0Q^{1/2-\gamma} \quad \mbox{and} \quad M:=C_1Q^{1/2-\gamma},
\end{equation}
it follows that 
\begin{equation*}
P(x)\ll 1+\Bigg| \frac{Q^{\varepsilon}}{L} \sum\limits_{|l|\le L} \sum\limits_{M_0\le m\le M}\alpha_l\beta_m e_r\left(l\sqrt{jm}\right) e\left(lf(m)\right)\Bigg|
\end{equation*}
for suitable $\alpha_l,\beta_m\ll 1$ and 
$$
f(x):=-\frac{Q^{3/4+\gamma/2}\sqrt{x}}{r}. 
$$
For $M_0\le x\le M$, it follows that
\begin{equation} \label{Fdef}
|f'(x)|\le F:=C_2\cdot \frac{Q^{1/2+\gamma}}{r} 
\end{equation}
for a suitable constant $C_2>0$. Applying Theorem \ref{bilinearboundnew}(i) and taking $||\boldsymbol{\alpha}||_2\ll L^{1/2}$ and $||\boldsymbol{\beta}||_{\infty}\ll 1$ into account, we deduce that 
\begin{equation} \label{Pxnew}
P(x)\ll \left(H^{-1/2}M+M^{1/2}r^{1/4}+L^{-1/2}M\right)(Qr)^{\varepsilon},
\end{equation}
provided that 
\begin{equation} \label{LMHcondis}
1\le L,M\le r\quad \mbox{and} \quad 1\le H\le \min\left\{\frac{1}{LF},M\right\}.
\end{equation}
We choose 
\begin{equation} \label{Hfix1}
H:=\frac{1}{LF}
\end{equation}
and will later make sure that this is consistent with the conditions on $H$ in \eqref{LMHcondis}, i.e. $1\le 1/(LF)\le M$. 
Recalling the definition of $F$ in \eqref{Fdef}, it then follows from \eqref{Pxnew} that 
\begin{equation*}
P(x)\ll \left(L^{1/2}MQ^{1/4+\gamma/2}r^{-1/2}+M^{1/2}r^{1/4}+L^{-1/2}M\right)(Qr)^{\varepsilon}.
\end{equation*}
Recalling the definition of $M$ in \eqref{Mdef} and $r\le Q$, this implies
\begin{equation} \label{Pxf}
P(x)\ll \left(L^{1/2}Q^{3/4-\gamma/2}r^{-1/2}+Q^{1/4-\gamma/2}r^{1/4}+L^{-1/2}Q^{1/2-\gamma}\right)Q^{\varepsilon}.
\end{equation}
Now we balance the first and last terms on the right-hand side of \eqref{Pxf}, choosing
\begin{equation} \label{Lchoice}
L:=\frac{r^{1/2}}{Q^{1/4+\gamma/2}},
\end{equation}
which gives
\begin{equation} \label{Pxfirst}
\begin{split}
P(x)\ll & \left(Q^{5/8-3\gamma/4}r^{-1/4}+Q^{1/4-\gamma/2}r^{1/4}\right)Q^{\varepsilon}
\ll \left(Q^{5/8}r^{-1/4}+Q^{1/4}r^{1/4}\right)Q^{\varepsilon}.
\end{split}
\end{equation}
Moreover, from \eqref{Ldef} and \eqref{Lchoice}, we deduce that 
$$
\delta=Q^{5/4+\gamma/2+\varepsilon}r^{1/2}
$$
This is consistent with \eqref{specialdeltarange} if 
\begin{equation} \label{rcond2}
r\le Q^{3/2-\gamma-2\varepsilon},
\end{equation}
which is a stronger condition than \eqref{rcond1}.
Using \eqref{Fdef} and \eqref{Lchoice}, the parameter $H$ defined in \eqref{Hfix1} takes the value
$$
H=\frac{r^{1/2}}{C_2Q^{1/4+\gamma/2}}.
$$
The conditions on $r$ in \eqref{rcond2} and on $L$, $M$ and $H$ in \eqref{LMHcondis} are now satisfied if 
\begin{equation*} 
\max\left\{1,C_2^2\right\}Q^{1/2+\gamma}\le r\le \min\left\{1,C_2^2\right\}Q^{3/2-\gamma-2\varepsilon} \quad \mbox{and} \quad \gamma\le \frac{1}{2}.
\end{equation*}
Hence, \eqref{Pxfirst} holds if $Q$ is large enough and 
\begin{equation} \label{rcond3}
Q^{1/2+\gamma+\varepsilon}\le r\le Q^{1-2\varepsilon} \quad \mbox{and} \quad \gamma\le \frac{1}{2}.
\end{equation}

\subsection{Completion of the proof of Theorem \ref{extendedrange}(i)}
The condition on $r$ in \eqref{rcond3} becomes more restrictive as $\gamma$ increases. Using the following lemma from \cite{BaiFirst}, we now establish an alternate bound for $P(x)$ which improves with increasing $\gamma$. We will use this bound if $r$ lies in the range $Q^{1/2+\varepsilon}\le r< Q^{1/2+\gamma+\varepsilon}$ or if $\gamma>1/2$.  

\begin{lemma} \label{al} We have 
\begin{equation*}
P\left(x\right)\ll \left(1+Q^2rz+Q^3\Delta\right)N^{\varepsilon}. 
\end{equation*}
\end{lemma} 
\begin{proof}
This follows from \cite[Lemma 5]{BaiFirst}.
\end{proof}

Recalling $Q^3=N$ and the relation between $z$ and $\gamma$ in \eqref{ranges}, Lemma \ref{al} implies 
\begin{equation*}
P\left(x\right)\ll \left(1+Q^{1/2-\gamma}\right)Q^{\varepsilon}. 
\end{equation*}
So if $r< Q^{1/2+\gamma+\varepsilon}$, then $Q^{\gamma}>rQ^{-1/2-\varepsilon}$ and hence 
\begin{equation} \label{Pxsecond}
P(x)\ll \left(1+Qr^{-1}\right)Q^{2\varepsilon},
\end{equation}
which is also valid if $\gamma>1/2$.
By the virtue of \eqref{Pxfirst} and \eqref{Pxsecond}, we deduce that 
$$
P(x)\ll \left(Qr^{-1}+Q^{5/8}r^{-1/4}+Q^{1/4}r^{1/4}\right)Q^{2\varepsilon}
$$
for all $r$ in the range $Q^{1/2+\varepsilon}\le r\le Q^{1-2\varepsilon}$ and all $z$ in the range in \eqref{supposedrange}. Noting that $Qr^{-1}\le Q^{5/8}r^{-1/4}$ if $r\ge Q^{1/2}$, the result of Theorem \ref{extendedrange}(i) follows.

To prove parts (ii) and (iii) of Theorem \ref{extendedrange}, we proceed similarly as above with the same choices of parameters, where we apply parts (ii) and (iii) of Theorem \ref{bilinearboundnew}, respectively. 

\section{Alternate treatment of prime moduli} \label{alter}
\subsection{Preparations} In this subsection, we collate results which we need for the proof of Theorem \ref{mainresult}. Without loss of generality, we assume that $p$ is an odd prime. 
Throughout the sequel, we let $U$ and $V$ be two parameters such that $1\le U,V<p$ and set 
$$
\mathcal{U}:=\{u\in (U/2,U] : u \mbox{ is a quadratic residue modulo } p\}
$$
and 
$$
\mathcal{V}:=(V/2,V]\cap \mathbb{N}.
$$
The following Proposition provides a lower bound for the cardinality of $\mathcal{U}$. 

\begin{proposition} \label{quadres}
If $U\ge p^{1/4+\varepsilon}$, then $\sharp \mathcal{U} \ge U/2+o(U)$. 
\end{proposition}

\begin{proof} This follows from 
$$
\sharp\mathcal{U}=\frac{1}{2}\cdot \left(U+\sum\limits_{U/2<u\le U} \left(\frac{u}{p}\right)\right)
$$
and the bound
$$
\sum\limits_{U/2<u\le U} \left(\frac{u}{p}\right)= o(U)
$$
for $U\ge p^{1/4+\varepsilon}$, which is a consequence of the Burgess bound for short character sums (see \cite[Theorem 12.6]{Bur}). 
\end{proof}

Thus, to ensure that $\mathcal{U}$ has a large cardinality, we will assume that  
\begin{equation} \label{Ucond}
U\ge p^{1/4+\varepsilon}
\end{equation}
throughout the sequel.  The following proposition provides estimates for the first and second moments of the quantity
\begin{equation} \label{nudef}
\nu(\lambda,\mu):=\sum\limits_{\substack{l\in \mathcal{X}\\ u\in \mathcal{U}\\ A-Y\le m\le A+2Y\\ (S(u)l,u^{-1}m)\equiv (\lambda,\mu)\bmod{p}}} |\alpha_l|,
\end{equation}
where $S(u)$ is a {\it fixed} square root of $u$ modulo $p$.
 
\begin{proposition}\label{moment}
We have 
\begin{equation} \label{firstm}
\sum\limits_{(\lambda,\mu)\in \mathbb{F}_p^2} \nu(\lambda,\mu)\ll ||\boldsymbol{\alpha}||_1UY
\end{equation}
and 
\begin{equation} \label{secondm}
\sum\limits_{(\lambda,\mu)\in \mathbb{F}_p^2} \nu(\lambda,\mu)^2\ll ||\boldsymbol{\alpha}||_{\infty}^2UXY\left(\frac{UX}{p}+1\right)p^{\varepsilon}. 
\end{equation}
\end{proposition}
\begin{proof} The estimate \eqref{firstm} is trivial. The left-hand side of \eqref{secondm} can be written in the form
\begin{equation*}
\begin{split}
\sum\limits_{(\lambda,\mu)\in \mathbb{F}_p^2} \nu(\lambda,\mu)^2= &\sum\limits_{\substack{l_1,l_2\in \mathcal{X}\\ u_1,u_2\in \mathcal{U}\\ A-Y\le m_1,m_2\le A+2Y\\ \left(S(u_1)l_1,u_1^{-1}m_1\right)\equiv  \left(S(u_2)l_2,u_2^{-1}m_2\right)\bmod{p}}} \left|\alpha_{l_1}\alpha_{l_2}\right|\\
\ll & ||\boldsymbol{\alpha}||_{\infty}^2\sum\limits_{\substack{l_1,l_2\in \mathcal{X}\\ u_1,u_2\in \mathcal{U}\\ A-Y\le m_1,m_2\le A+2Y\\ \left(S(u_1)l_1,u_1^{-1}m_1\right)\equiv  \left(S(u_2)l_2,u_2^{-1}m_2\right)\bmod{p}}} 1.
\end{split}
\end{equation*}
If $u_1$, $u_2$ and $m_1$ are fixed, then the congruence $u_1^{-1}m_1\equiv  u_2^{-1}m_2\bmod{p}$ has at most $O(1)$ solutions $m_2\in [A-Y,A+2Y]$. It follows that 
 \begin{equation}\label{lhs}
\begin{split}
\sum\limits_{(\lambda,\mu)\in \mathbb{F}_p^2} \nu(\lambda,\mu)^2\ll &
||\boldsymbol{\alpha}||_{\infty}^2Y\sum\limits_{\substack{l_1,l_2\in \mathcal{X}\\ u_1,u_2\in \mathcal{U}\\ 
S(u_1)l_1\equiv S(u_2)l_2\bmod{p}}} 1
\ll ||\boldsymbol{\alpha}||_{\infty}^2Y\sum\limits_{\substack{l_1,l_2\in \mathcal{X}\\ u_1,u_2\in \mathcal{U}\\ 
u_1l_1^2\equiv u_2l_2^2\bmod{p}}} 1.
\end{split}
\end{equation}
Now \cite[Theorem 4.1]{CiShZ} provides us with the bound
\begin{equation} \label{CiShZeq}
\sum\limits_{\substack{l_1,l_2\in \mathcal{X}\\ u_1,u_2\in \mathcal{U}\\ 
u_1l_1^2\equiv u_2l_2^2\bmod{p}}} 1 \ll \frac{X^2U^2}{p}+XUp^{\varepsilon}.
\end{equation}
Combining \eqref{lhs} and \eqref{CiShZeq} gives \eqref{secondm}, completing the proof of Proposition \ref{moment}.
\end{proof} 
\begin{remark} \label{rema} We point out that \cite[inequality (5.5)]{BaSh} claims the same bound as \eqref{secondm} for a related quantity with the term $UY/p$ in place of $UX/p$ on the right-hand side. We believe that this is a mistake. It does not substantially change the final result in \cite{BaSh}, though. \end{remark}

\begin{proposition}\label{solutions}
Let $\tilde{\mathcal{V}}$ be a subset of $\mathcal{V}$. Then for any $\xi\in \mathbb{R}$, we have 
\begin{equation*} 
\sum\limits_{(\lambda,\mu)\in \mathbb{F}_p^2} \left| \sum\limits_{v\in \tilde{\mathcal{V}}} e_p\left(\lambda\sqrt{j(\mu-v)}\right)e(\xi v)\right|^{2n} \ll \left(V^{2n}+pV^n\right)p. 
\end{equation*}
\end{proposition}

\begin{proof} Expanding the $2n$-th power and moving in the summation over $\mu$ and $\lambda$, we have
\begin{equation} \label{tran} 
\begin{split}
& \sum\limits_{(\lambda,\mu)\in \mathbb{F}_p^2} \left| \sum\limits_{v\in \tilde{\mathcal{V}}} e_p\left(\lambda\sqrt{j(\mu-v)}\right)e(\xi v)\right|^{2n}\\ = &  \sum\limits_{(v_1,...,v_{2n})\in \tilde{\mathcal{V}}^{2n}} \sum\limits_{\mu\in \mathbb{F}_p}\sum_{\lambda\in \mathbb{F}_p} e_p\left(\lambda\left(\sum\limits_{i=1}^n\sqrt{j(\mu-v_i)}-\sum\limits_{i=n+1}^{2n}\sqrt{j(\mu-v_i)} \right)\right)e\left(\xi \left(\sum\limits_{i=1}^n v_i-\sum\limits_{i=n+1}^{2n} v_{i}\right)\right)\\
 = & p \sum\limits_{(v_1,...,v_{2n})\in \tilde{\mathcal{V}}^{2n}} \sum\limits_{\substack{\mu\in \mathbb{F}_p\\ C(v_1,...,v_{2n};\mu)}} e\left(\xi \left(\sum\limits_{i=1}^n v_i-\sum\limits_{i=n+1}^{2n} v_{i}\right)\right)\\
 \ll & p\sum\limits_{(v_1,...,v_{2n})\in \tilde{\mathcal{V}}^{2n}
} \sum\limits_{\substack{\mu\in \mathbb{F}_p\\  C(v_1,...,v_{2n};\mu)}} 1,
\end{split}
\end{equation}
where $C(v_1,...,v_{2n};\mu)$ denotes the congruence
$$
\sum\limits_{i=1}^n\sqrt{j(\mu-v_i)}\equiv \sum\limits_{i=n+1}^{2n}\sqrt{j(\mu-v_i)}\bmod{p}.
$$
We recall that, by our general convention in section 1, the modular square roots above are not unique but only unique up to sign. Now for any $k\in \mathbb{F}_p$, fix a square root $S(k)=\sqrt{k}$ in the algebraic closure $\hat{\mathbb{F}}_p$ of $\mathbb{F}_p$. Given $(v_1,...,v_{2n})\in \tilde{\mathcal{V}}^{2n}$, if $\mu$ solves the above congruence $C(v_1,...,v_{2n};\mu)$ for some choice of signs of the modular square roots, then  
\begin{equation} \label{poly}
N(\mu,v_1,...,v_{2n})\equiv 0 \bmod{p},
\end{equation}
where we define
$$
N(x,y_1,...,y_{2n}):=\prod\limits_{(a_1,...,a_{2n})\in \{0,1\}^{2n}} F(x,y_1,...,y_{2n})
$$
with
$$
F(x,y_1,...,y_{2n}):=\sum\limits_{i=1}^{n} (-1)^{a_i}S(j(x-y_i))-\sum\limits_{i=n+1}^{2n} (-1)^{a_i}S(j(x-y_i)).
$$
Since $N(x,y_1,...,y_{2n})$ equals the norm of 
$$
F(x,y_1,...,y_{2n})\in {\mathbb{F}}_p(x,y_1,...,y_{2n},\sqrt{j(x-y_1)},...,\sqrt{j(x-y_{2n})})
$$ 
over $\mathbb{F}_p(x,y_1,...,y_{2n})$, it follows that $N(x,y_1,...,y_{2n})$ is a polynomial in $\mathbb{F}_p[x,y_1,...,y_{2n}]$. Hence, 
$$
N_{y_1,...,y_{2n}}(x):=N(x,y_1,...,y_{2n})
$$
is a polynomial in $\mathbb{F}_p[x]$ with coefficients in $\mathbb{F}_p[y_1,...,y_{2n}]$. If $p$ is large enough compared to $n$ and $(v_1,...,v_{2n})\in \tilde{\mathcal{V}}^{2n}$, then the polynomial $N_{v_1,...,v_{2n}}(x)\in \mathbb{F}_p[x]$ vanishes identically if and only if for every $v\in \tilde{\mathcal{V}}$, the number of $i\in \{1,...,2n\}$ with $v_i=v$ is even. (Here we use the fact that $\sqrt{jx},\sqrt{j(x-1)},...,\sqrt{j(x-(r-1))}$ are linearly independent over $\mathbb{F}_p(x)$.) If $N_{v_1,...,v_{2n}}(x)$ vanishes identically,  then \eqref{poly} holds trivially, i.e., every $\mu\in \mathbb{F}_p$ solves this congruence. Otherwise, the number of solutions $\mu$ of the said congruence \eqref{poly} is less than or equal to the degree of $N_{v_1,...,v_{2n}}(x)$, which is bounded in terms of $n$. We conclude that 
\begin{equation*}
\sum\limits_{(v_1,...,v_{2n})\in \tilde{\mathcal{V}}^{2n}
} \sum\limits_{\substack{\mu\in \mathbb{F}_p\\  C(v_1,...,v_{2n};\mu)}} 1\ll_n V^{2n}+pW,
\end{equation*}  
where $W$ is the number of $2n$-tuples $(v_1,...,v_{2n})\in \tilde{\mathcal{V}}^{2n}$ such that 
$\sharp\{i\in \{1,...,2n\}: v_i=v\}$ is even for every $v\in \tilde{\mathcal{V}}$. If $\sharp\tilde{\mathcal{V}}=k$, then this number is given by 
$$
W=\sum\limits_{w=1}^n \sum\limits_{\substack{c_1\le c_2\le...\le c_w\\ c_1+\cdots+c_w=n}} \binom{2n}{(2c_1),...,(2c_w)} \cdot k(k-1)\cdots (k-w+1)
$$ 
and satisfies the bound
$$
W\ll_n k^n. 
$$
It follows that 
\begin{equation*}
\sum\limits_{(v_1,...,v_{2n})\in \tilde{\mathcal{V}}^{2n}
} \sum\limits_{\substack{\mu\in \mathbb{F}_p\\  C(v_1,...,v_{2n};\mu)}} 1\ll_n V^{2n}+pV^n.
\end{equation*}  
Combining this with \eqref{tran} proves the desired estimate. 
\end{proof} 

\subsection{Proof of Theorem \ref{mainresult}}
Let $\mathcal{I}=(L_0,L_0+L]$, where, without loss of generality, $L_0$ is an integer. Throughout this section, we write 
\begin{equation*}
\Sigma:=\Sigma_f(p,j,\mathcal{X},\mathcal{Y},\boldsymbol{\alpha}).
\end{equation*}
We first observe that we may assume that $L_0=0$ without loss of generality since  
$$
\sum\limits_{l\in \mathcal{X}}\sum\limits_{m\in \mathcal{Y}} \alpha_le_p\left(l\sqrt{jm}\right)e(lf(m))=
\sum\limits_{\ell\in \tilde{\mathcal{X}}} \sum\limits_{m\in \mathcal{Y}} \tilde{\alpha}_{\ell}e_p\left(\ell\sqrt{jm}\right)e(\ell f(m))
$$
with 
$$
\tilde{\mathcal{X}}:=\{l-L_0: l\in \mathcal{X}\} \quad \mbox{and} \quad \tilde{\alpha}_{\ell}:=\alpha_{\ell+L_0} e_p\left(L_0\sqrt{jm}\right)e(L_0f(m)).
$$
Further, we will assume that 
\begin{equation} \label{UVcond}
UV\le \min\left\{\frac{1}{FL},Y\right\},
\end{equation}
where we recall that 
$$
F=\sup\limits_{A-2Y\le y\le A+2Y} |f'(y)|.
$$ 
We start by writing
\begin{equation*}
\Sigma=\frac{1}{\sharp\mathcal{U}\sharp \mathcal{V}}\cdot \sum\limits_{l\in \mathcal{X}} \sum\limits_{u\in \mathcal{U}} \sum\limits_{v\in \mathcal{V}} \sum\limits_{\substack{A-Y\le m\le A+2Y\\ m-uv\in \mathcal{Y}}} \alpha_l e\left(lf(m-uv)\right)e_p\left(l\sqrt{j(m-uv)}\right). 
\end{equation*}
Following the transformations in \cite[equations(4.3) and(4.4)]{FM}, as in \cite{BaSh}, and using Proposition \ref{quadres}, we have 
\begin{equation*}
\begin{split}
\Sigma\ll & \frac{\log p}{UV}\cdot \sum\limits_{l\in \mathcal{X}} |\alpha_l| \sum\limits_{u\in \mathcal{U}} \sum\limits_{A-Y\le m\le A+2Y} \left|\sum\limits_{v\in \mathcal{V}} e\left(lf(m-uv)\right)e_p\left(l\sqrt{j(m-uv)}\right)e(\xi v)\right|
\end{split}
\end{equation*}
for some $\xi\in \mathbb{R}$. 
Under the assumption $L_0=0$ and condition \eqref{UVcond}, we may remove the analytic term $e\left(lf(m-uv)\right)$ using partial summation, obtaining 
\begin{equation*}
\begin{split}
\Sigma\ll & \frac{\log p}{UV}\cdot \sup\limits_{V/2\le y\le V} \sum\limits_{l\in \mathcal{X}} |\alpha_l|  \sum\limits_{u\in \mathcal{U}} \sum\limits_{A-Y\le m\le A+2Y}   \left|\sum\limits_{v\in \mathcal{V}(y)} e_p\left(l\sqrt{j(m-uv)}\right)e(\xi v)\right|,
\end{split}
\end{equation*}
where 
$$
\mathcal{V}(y):=\{v\in \mathcal{V}: v\le y\}.
$$
Now we write
$$
e_p\left(l\sqrt{j(m-uv)}\right)=e_p\left(\lambda\sqrt{j(\mu-v)}\right)
$$
with 
$$
\lambda=S(u)l \quad \mbox{and} \quad \mu=u^{-1}m,
$$
where $S(u)$ is a {\it fixed} square root of $u$ modulo $p$.
Then we see that
\begin{equation} \label{using}
\Sigma\ll \frac{\log p}{UV}\cdot \sup\limits_{V/2\le y\le V} \Sigma(y)
\end{equation}
with
\begin{equation*}
\Sigma(y):=\sum\limits_{(\lambda,\mu)\in \mathbb{F}_p^2} \nu(\lambda,\mu)\cdot \left|\sum\limits_{v\in \mathcal{V}(y)} e_p\left(\lambda\sqrt{j(\mu-v)}\right)e(\xi v)\right|,
\end{equation*}
where $\nu(\lambda,\mu)$ is defined as in \eqref{nudef}. 
Using H\"older's inequality, we deduce that for every $n\in \mathbb{N}$,
\begin{equation*}
\begin{split}
\Sigma(y)\ll & \left(\sum\limits_{(\lambda,\mu)\in \mathbb{F}_p^2} \nu(\lambda,\mu)\right)^{1-1/n}\left(\sum\limits_{(\lambda,\mu)\in \mathbb{F}_p^2} \nu(\lambda,\mu)^2\right)^{1/2n} \left(\sum\limits_{(\lambda,\mu)\in \mathbb{F}_p^2} \left| \sum\limits_{v\in \mathcal{V}(y)} e_p\left(\lambda\sqrt{j(\mu-v)}\right)e(\xi v)\right|^{2n}\right)^{1/2n}.
\end{split}
\end{equation*}
Applying Propositions \ref{moment} and \ref{solutions}, and using \eqref{using}, we obtain the bound 
\begin{equation*}
\begin{split}
\Sigma\ll_n &
\frac{\log p}{UV}\cdot \left(||\boldsymbol{\alpha}||_1||UY\right)^{1-1/n}
\left(||\boldsymbol{\alpha}||_{\infty}^2UXY\left(\frac{UX}{p}+1\right)r^{\varepsilon}\right)^{1/2n}  \left(V^{2n}+pV^n\right)^{1/2n}p^{1/2n}\\
\ll & ||\boldsymbol{\alpha}||_1^{1-1/n}||\boldsymbol{\alpha}||_{\infty}^{1/n}
X^{1/2n}Y^{1-1/2n}\left(Xp^{-1}+U^{-1}\right)^{1/2n}\left(1+pV^{-n}\right)^{1/2n}p^{1/2n+\varepsilon}. 
\end{split}
\end{equation*}
Choosing
$$
V:=p^{1/n} \quad \mbox{and} \quad U:=\min\left\{\frac{1}{LFp^{1/n}},\frac{Y}{p^{1/n}}\right\},
$$
the claimed bound \eqref{mainbound} follows. We note that under the assumptions \eqref{Lcond} and \eqref{Ycond} in Theorem \ref{mainresult}, these choices of $U$ and $V$ are consistent with \eqref{Ucond} and \eqref{UVcond}. 
Hence, the proof of Theorem \ref{mainresult} is complete.

\end{document}